\documentclass[a4paper,twoside]{article}
\usepackage{a4}
\usepackage{amssymb}
\usepackage{amsmath}
\usepackage{upref}
\usepackage[active]{srcltx}
\usepackage[pagebackref,colorlinks,citecolor=blue,linkcolor=blue,urlcolor=blue]{hyperref}
\usepackage[dvipsnames]{color}
\allowdisplaybreaks[2] 
\newcount\minutes \newcount\hours
\hours=\time
\divide\hours 60
\minutes=\hours
\multiply\minutes -60
\advance\minutes \time
\newcommand{\klockan}{\the\hours:{\ifnum\minutes<10 0\fi}\the\minutes}
\newcommand{\tid}{\today\ \klockan}
\newcommand{\prtid}{\smash{\raise 10mm \hbox{\LaTeX ed \tid}}}
\renewcommand{\prtid}{}
\makeatletter
\def\sectionmark#1{} 
\def\subsectionmark#1{}
\newcommand{\sectnr}{\ifnum \c@secnumdepth >\z@
                 \thesection.\hskip 1em\relax \fi}
\def\@evenhead{\footnotesize\rm\thepage\hfil\leftmark\hfil\llap{\prtid}}
\def\@oddhead{\footnotesize\rm\rlap{\prtid}\hfil\rightmark\hfil\thepage}
\def\tableofcontents{\section*{Contents} 
 \@starttoc{toc}}
\makeatother
\makeatletter
\def\@biblabel#1{#1.}
\makeatother
\makeatletter
\let\Thebibliography=\thebibliography
\renewcommand{\thebibliography}[1]{\def\@mkboth##1##2{}\Thebibliography{#1}
\addcontentsline{toc}{section}{References}
\frenchspacing 
\setlength{\@topsep}{0pt}
\setlength{\itemsep}{0pt}%
\setlength{\parskip}{0pt plus 2pt}%
}
\makeatother
\makeatletter
\def\mdots@{\mathinner.\nonscript\!.%
 \ifx\next,.\else\ifx\next;.\else\ifx\next..\else
 \nonscript\!\mathinner.\fi\fi\fi}
\let\ldots\mdots@
\let\cdots\mdots@
\let\dotso\mdots@
\let\dotsb\mdots@
\let\dotsm\mdots@
\let\dotsc\mdots@
\def\vdots{\vbox{\baselineskip2.8\p@ \lineskiplimit\z@
    \kern6\p@\hbox{.}\hbox{.}\hbox{.}\kern3\p@}}
\def\ddots{\mathinner{\mkern1mu\raise8.6\p@\vbox{\kern7\p@\hbox{.}}%
    \raise5.8\p@\hbox{.}\raise3\p@\hbox{.}\mkern1mu}}
\makeatother
\makeatletter
\let\Enumerate=\enumerate
\renewcommand{\enumerate}{\Enumerate%
\setlength{\itemsep}{0pt}%
\setlength{\parskip}{0pt plus 1pt}%
\renewcommand{\theenumi}{\textup{(\alph{enumi})}}%
\renewcommand{\labelenumi}{\theenumi}%
}
\makeatother
\makeatletter
\def\@seccntformat#1{\csname the#1\endcsname.\quad}
\makeatother
\RequirePackage{ifthen}
\newcommand{\authortitle}[3]{\author{#1}\title{#2}%
   \ifthenelse{\equal{#3}{}}{\markboth{#1}{#2}}{\markboth{#1}{#3}}}
\newcommand{\art}[6]{{\sc #1, \rm #2, \it #3 \bf #4 \rm (#5), \mbox{#6}.}}
\newcommand{\auth}[2]{{#1, #2.}}

\newcommand{\book}[3]{{\sc #1, \it #2, \rm #3.}}
\newcommand{\AND}{{\rm and }}

\RequirePackage{amsthm}
\newtheoremstyle{descriptive}%
  {\topsep}   
  {\topsep}   
  {\rmfamily} 
  {}          
  {\bfseries} 
  {.}         
  { }         
  {}          
\newtheoremstyle{propositional}%
  {\topsep}   
  {\topsep}   
  {\itshape}  
  {}          
  {\bfseries} 
  {.}         
  { }         
  {}          
\newtheoremstyle{remarkstyle}%
  {\topsep}   
  {\topsep}   
  {\rmfamily}  
  {}          
  {\itshape} 
  {.}         
  { }         
  {}          
\theoremstyle{propositional}
\newtheorem{thm}{Theorem}[section]
\newtheorem{prop}[thm]{Proposition}
\newtheorem{lem}[thm]{Lemma}
\newtheorem{cor}[thm]{Corollary}
\newtheorem*{ass}{General Assumptions}
\theoremstyle{descriptive}
\newtheorem{deff}[thm]{Definition}
\newtheorem{example}[thm]{Example}
\theoremstyle{remarkstyle}
\newtheorem{remark}[thm]{Remark}
\makeatletter
\renewenvironment{proof}[1][\proofname]{\par
  \pushQED{\qed}%
  \normalfont
  \trivlist
  \item[\hskip\labelsep
        \itshape
    #1\@addpunct{.}]\ignorespaces
}{%
  \popQED\endtrivlist\@endpefalse
}
\makeatother
\newcommand{\setm}{\setminus}
\renewcommand{\subsetneq}{\varsubsetneq}
\renewcommand{\emptyset}{\varnothing}
\def\vint{\mathop{\mathchoice%
          {\setbox0\hbox{$\displaystyle\intop$}\kern 0.22\wd0%
           \vcenter{\hrule width 0.6\wd0}\kern -0.82\wd0}%
          {\setbox0\hbox{$\textstyle\intop$}\kern 0.2\wd0%
           \vcenter{\hrule width 0.6\wd0}\kern -0.8\wd0}%
          {\setbox0\hbox{$\scriptstyle\intop$}\kern 0.2\wd0%
           \vcenter{\hrule width 0.6\wd0}\kern -0.8\wd0}%
          {\setbox0\hbox{$\scriptscriptstyle\intop$}\kern 0.2\wd0%
           \vcenter{\hrule width 0.6\wd0}\kern -0.8\wd0}}%
          \mathopen{}\int}

\newcommand{\Cp}{{C_p}}
\newcommand{\CpY}{{C_p^Y}}
\newcommand{\Cpmu}{{C_{p,\mu}}}
\DeclareMathOperator{\diam}{diam}
\DeclareMathOperator{\capp}{cap}
\newcommand{\cp}{\capp_p}
\DeclareMathOperator{\dist}{dist}
\DeclareMathOperator{\Lip}{Lip}
\newcommand{\Lipc}{{\Lip_c}}
\DeclareMathOperator{\spt}{supp}
\newcommand{\supp}{\spt}
\DeclareMathOperator*{\essliminf}{ess\,lim\,inf}
\let\Re\undefined \DeclareMathOperator{\Re}{Re}
\newcommand{\bdry}{\partial}
\newcommand{\bdy}{\bdry}
\newcommand{\loc}{_{\rm loc}}
\DeclareMathOperator{\para}{par}
{\catcode`p =12 \catcode`t =12 \gdef\eeaa#1pt{#1}}      
\def\accentadjtext#1{\setbox0\hbox{$#1$}\kern   
                \expandafter\eeaa\the\fontdimen1\textfont1 \ht0 }
\def\accentadjscript#1{\setbox0\hbox{$#1$}\kern 
                \expandafter\eeaa\the\fontdimen1\scriptfont1 \ht0 }
\def\accentadjscriptscript#1{\setbox0\hbox{$#1$}\kern   
                \expandafter\eeaa\the\fontdimen1\scriptscriptfont1 \ht0 }
\def\accentadjtextback#1{\setbox0\hbox{$#1$}\kern       
                -\expandafter\eeaa\the\fontdimen1\textfont1 \ht0 }
\def\accentadjscriptback#1{\setbox0\hbox{$#1$}\kern     
                -\expandafter\eeaa\the\fontdimen1\scriptfont1 \ht0 }
\def\accentadjscriptscriptback#1{\setbox0\hbox{$#1$}\kern 
                -\expandafter\eeaa\the\fontdimen1\scriptscriptfont1 \ht0 }
\def\itoverline#1{{\mathsurround0pt\mathchoice
        {\rlap{$\accentadjtext{\displaystyle #1}
                \accentadjtext{\vrule height1.593pt}
                \overline{\phantom{\displaystyle #1}
                \accentadjtextback{\displaystyle #1}}$}{#1}}
        {\rlap{$\accentadjtext{\textstyle #1}
                \accentadjtext{\vrule height1.593pt}
                \overline{\phantom{\textstyle #1}
                \accentadjtextback{\textstyle #1}}$}{#1}}
        {\rlap{$\accentadjscript{\scriptstyle #1}
                \accentadjscript{\vrule height1.593pt}
                \overline{\phantom{\scriptstyle #1}
                \accentadjscriptback{\scriptstyle #1}}$}{#1}}
        {\rlap{$\accentadjscriptscript{\scriptscriptstyle #1}
                \accentadjscriptscript{\vrule height1.593pt}
                \overline{\phantom{\scriptscriptstyle #1}
                \accentadjscriptscriptback{\scriptscriptstyle #1}}$}{#1}}}}
\def\itunderline#1{{\mathsurround0pt\mathchoice
        {\rlap{$\underline{\phantom{\displaystyle #1}
                \accentadjtextback{\displaystyle #1}}$}{#1}}
        {\rlap{$\underline{\phantom{\textstyle #1}
                \accentadjtextback{\textstyle #1}}$}{#1}}
        {\rlap{$\underline{\phantom{\scriptstyle #1}
                \accentadjscriptback{\scriptstyle #1}}$}{#1}}
        {\rlap{$\underline{\phantom{\scriptscriptstyle #1}
                \accentadjscriptscriptback{\scriptscriptstyle #1}}$}{#1}}}}
\newcommand{\al}{\alpha}
\newcommand{\Om}{\Omega}
\renewcommand{\phi}{\varphi}
\newcommand{\eps}{\varepsilon}
\newcommand{\de}{\delta}
\newcommand{\ga}{\gamma}
\newcommand{\la}{\lambda}
\newcommand{\La}{\Lambda}
\newcommand{\p}{{$p\mspace{1mu}$}}
\newcommand{\R}{\mathbf{R}}
\newcommand{\C}{\mathbf{C}}
\newcommand{\eR}{{\overline{\R}}}
\newcommand{\Np}{N^{1,p}}
\newcommand{\Dp}{D^{p}}
\newcommand{\Nploc}{N^{1,p}\loc}
\newcommand{\A}{\ensuremath{\mathcal{A}}}%
\newcommand{\UU}{\mathcal{U}}%
\newcommand{\clOm}{{\overline{\Om}}}
\newcommand{\Ga}{\Gamma}
\newcommand{\Lploc}{L^{p}\loc}
\newcommand{\clG}{\itoverline{G}}     
\newcommand{\uP}{\itoverline{P}}     
\newcommand{\lP}{\itunderline{P}} 
\newcommand{\ut}{\tilde{u}}
\newcommand{\zpm}{z^\pm}
\newcommand{\zpl}{z^+}
\newcommand{\zmin}{z^-}
\newcommand{\apm}{a^\pm}
\newcommand{\apl}{a^+}
\newcommand{\amin}{a^-}
\newcommand{\Ipm}{I^\pm}
\newcommand{\Ipl}{I^+}
\newcommand{\Imin}{I^-}
\newcommand{\Xplus}{X_+}  
\newcommand{\clB}{\itoverline{B}}
\newcommand{\bdystar}{\partial^*}
\newcommand{\clV}{\overline{V}}
\newcommand{\simge}{\gtrsim}
\newcommand{\simle}{\lesssim}
\newcommand{\binfty}{{\boldsymbol{\infty}}}
\newcommand{\Rn}{\mathbf{R}^n}

\numberwithin{equation}{section}
\newcommand{\imp}{\ensuremath{\mathchoice{\quad \Longrightarrow \quad}{\Rightarrow}
                {\Rightarrow}{\Rightarrow}}}

\begin{document}
%
%

\authortitle{Anders Bj\"orn and Jana Bj\"orn}
{Liouville theorems and removable sets for  \\
bounded \p-harmonic and   quasiharmonic  functions \\
  on metric spaces under local assumptions}
{Liouville theorems and removable 
sets for bounded \p-harmonic 
functions}

\author{
Anders Bj\"orn \\
\it\small Department of Mathematics, Link\"oping University, SE-581 83 Link\"oping, Sweden\\
\it \small anders.bjorn@liu.se, ORCID\/\textup{:} 0000-0002-9677-8321
\\
\\
Jana Bj\"orn \\
\it\small Department of Mathematics, Link\"oping University, SE-581 83 Link\"oping, Sweden\\
\it \small jana.bjorn@liu.se, ORCID\/\textup{:} 0000-0002-1238-6751
}

\date{Preliminary version, \today}
\date{}
\maketitle

\noindent {\small {\bf Abstract}.  
For connected proper metric spaces $X$, equipped with a locally doubling measure
supporting a local \p-Poincar\'e inequality, we completely characterize which compact sets $K$ with positive
capacity are removable for  bounded \p-harmonic functions, $p>1$.
Similar results are proved also for bounded quasiharmonic functions.
The characterization is both in geometric and analytic terms.
In particular, removability is shown to be equivalent to the validity of a Liouville type theorem
in $X\setminus K$.
Properties such as local connectedness, sequential annular quasiconvexity, concentration of capacity
and \p-parabolicity are identified as crucial for removability.
Along the way, we give a rather elementary proof of the Liouville theorem 
for quasisuperharmonic functions in \p-parabolic  spaces.
Our results apply in particular to manifolds and $\mathbf{R}^n$ equipped with (locally) 
\p-admissible weights.
}



\medskip

\noindent {\small \emph{Funding}: 
A.~B. resp.\ J.~B. were supported by the Swedish Research Council,
grants 2020-04011 and 2024-04095 resp.\ 2022-04048.
}



\section{Introduction}

In this paper, we study removable singularities 
for bounded \p-harmonic and quasiharmonic functions
in rather general settings.
We assume that  $X$ is a proper connected metric space
equipped with a locally doubling measure $\mu$ that
supports a local \p-Poincar\'e inequality, where $1<p<\infty$.
Let $K \subsetneq \Om$ be a compact subset of an open set $\Om\subset X$.

We say that $K$ is \emph{removable for bounded \p-harmonic 
functions} in $\Om \setm K$
if  every bounded \p-harmonic function in $\Om \setm K$
has a bounded \p-harmonic  extension  to $\Om$.
The definition for $Q$-quasiharmonic functions is similar.
Here, a continuous function $u$ is \emph{$Q$-quasiharmonic}
if it quasiminimizes (up to a factor $Q \ge 1$)
the \p-energy integral 
\[ 
      \int g^p_u \, d\mu,
\] 
see Definition~\ref{def-quasimin}.
When $Q=1$, the function $u$ is called \emph{\p-harmonic}.

It is well known that sets with Sobolev \p-capacity 
$\Cp(K)=0$ are always removable (see Theorem~\ref{thm-removability-qharm}).
At the same time, even in unweighted $\Rn$, $n \ge 2$,
sets with $\Cp(K)>0$ can be removable,
namely when $\Om=\Rn$, $p>n$ and $K$ is a singleton.
A full characterization of (relatively closed) removable sets 
in $\Rn$, $n \ge 2$, equipped with a \p-admissible weight
was given in Theorem~1.3 in~\cite{ABRnremove}.

The main focus of this paper is on removability of
compact sets with positive \p-capacity.
The following  simple one-dimensional examples    
illustrate that seemingly similar spaces and sets can 
have different removability properties.
They also show that the measure on $X$ plays an important role
for removability.
All these examples satisfy our general assumptions.
Further examples of ``big'' removable sets are given
in Section~\ref{sect-examples}.

\begin{example} \label{ex-interval-new}
(a)
Clearly, no compact set $K\ne\emptyset$ can be removable for 
$\R\setm K$ because it disconnects the space and 
nonconstant piecewise constant functions
in $\R\setm K$ are \p-harmonic but do not have 
bounded \p-harmonic extensions to $\R$,
because every bounded \p-harmonic function on $\R$ is constant.

(b)
On the other hand, both the singleton $K=\{0\}$ and the interval
$K=[0,1]$ have $\Cp(K)>0$ but
are removable for $X\setm K$ when $X=[0,\infty)$
or $X=[0,2]$.
Indeed, 
any bounded \p-harmonic 
function in $X \setm K$ must be constant,
and extends trivially to $X$.
The same argument shows that $K$ is removable
for bounded $Q$-quasiharmonic 
function in $X \setm K$.

(c)
At the same time, if $X=[0,\infty)$ is equipped with the
weighted measure $d\mu=w\,dx$, where $w(x)=\max\{x^\al,1\}$, $\al>p-1$,
then the function
\[
u(x)=\int_0^{x} w^{1/(1-p)}\,dt
\] 
is bounded and \p-harmonic in $(0,\infty)$
with respect to $\mu$, but its unique continuous extension to $X=[0,\infty)$
is not \p-harmonic.
So $K=\{0\}$ is not removable in this case.
Note that $X$  (equipped with $\mu$) is a \p-hyperbolic space.
\end{example}

\begin{example} \label{ex-circle-new}
Using complex notation, let $X=\{z \in \C : |z|=1\}$ 
equipped with the one-dimensional Lebesgue measure.
Then $u(z)=\arg z$ (the principal branch) 
is a nonconstant
bounded \p-harmonic function in $X \setm K$,
both for $K=\{-1\}$ and $K=\{z : \Re z \le 0\}$.
Since all bounded \p-harmonic functions in $X$ are constant,
by the strong maximum principle, we see that
$K$ is not removable for bounded \p-harmonic
functions in $X \setm K$.
Note  that $X \setm K$ is not locally connected at $z=-1$
when $K=\{-1\}$, but that $X \setm K$ is locally 
connected at the two points in $\bdy K = \{\pm i\}$
when $K=\{z : \Re z \le 0\}$.
\end{example}

Our first two main results 
(Theorem~\ref{thm-main-char-intro} and Proposition~\ref{prop-not-locconn})
show that the properties appearing in the above examples
(disconnectedness, \p-hyperbolicity, lack of local connectedness, 
and that the capacity of $\bdy K$ is not concentrated at one point)
are crucial obstacles for removability
(of compact sets with positive capacity).
The following result  
characterizes compact removable sets with positive capacity.

\begin{thm} \label{thm-main-char-intro}
Assume that   $X$ is a proper connected metric space
equipped with a locally doubling measure $\mu$ that
supports a local \p-Poincar\'e inequality, where $1<p<\infty$.

Let  $\Om\subset X$ be  open, $Q\ge1$ and
$K \subsetneq \Om$ be compact with $\Cp(K)>0$.
Then the following are equivalent for $G:=\Om \setm K$\/\textup:
\begin{enumerate}
\item \label{f-rem}
$K$ is removable for bounded $Q$-quasiharmonic functions in $G$.
\item \label{f-weak}
Every bounded $Q$-quasiharmonic 
function in $G$ has a quasiharmonic extension to $\Om$
\textup(not necessarily with the same $Q$\textup).
\item \label{f-Li}
The Liouville theorem holds for bounded $Q$-quasiharmonic functions in $G$,
i.e.\ all bounded $Q$-quasiharmonic function in $G$ are constant.
In particular, $G$ is connected.
\item \label{f-char}
All of the following conditions hold\/\textup:
\begin{enumerate}
\renewcommand{\theenumii}{\textup{(\roman{enumii})}}%
\renewcommand{\labelenumii}{\theenumii}%
\makeatletter
\renewcommand\p@enumii{}
\makeatother
\item \label{rs-par}
$X$ is \p-parabolic or bounded,
\item \label{rs1+2}
there is $x_0 \in \bdy K$ such that $\Cp(\bdy G \setm \{x_0\})=0$,
\item \label{rs-lim}
the limit 
\begin{equation}   \label{eq-lim-ex}
\lim_{G \ni y \to x_0} u(y) \quad \text{exists}
\end{equation}
for every bounded $Q$-quasiharmonic function $u$ in~$G$.
\end{enumerate}
\end{enumerate}
In particular, with $Q=1$, the 
equivalences hold for \p-harmonic functions.
\end{thm}

Theorem~\ref{thm-LL-F} shows that $X$ in \ref{rs-par} 
can equivalently be replaced by $G$. 
The geometric conditions~\ref{rs-par}--\ref{rs1+2} in~\ref{f-char}
are easy to understand and independent of $Q$,
but the validity of \ref{rs-lim} is less clear.
We will show that, with
$G_0=\clG \setm \{x_0\}$ (which is open by \ref{rs1+2}
and Lemma~\ref{lem-G_0-open}),
\begin{align}
  &\text{$G_0$ is sequentially annularly quasiconvex at $x_0$ 
                 (Definition~\ref{def:ann-qcvx})} \nonumber\\
  & \qquad \imp \text{\ref{rs-lim}} 
  \imp \text{$G_0$ is locally connected at $x_0$},
\label{eq-seq=>lim=>loc-conn}
\end{align}
see Lemma~\ref{lem-lim-ann-quasiconvex} and Proposition~\ref{prop-not-locconn}.
We do not know if \ref{rs-lim} is
independent of $Q$.

Examples~\ref{ex-factorial2} and~\ref{ex-factorial3}, which may be
of independent interest, show that neither implication 
in~\eqref{eq-seq=>lim=>loc-conn} can be reversed.
To our knowledge, they seem to be the first examples 
with a globally doubling measure supporting a global \p-Poincar\'e inequality
such that $X \setm \{x_0\}$ is locally connected, 
but \emph{not annularly quasiconvex} at  $x_0$.
Examples~\ref{ex-hyp-R} and~\ref{ex-two-R} show that 
Theorem~\ref{thm-main-char-intro} does not hold if the
compact set $K$ therein is replaced by an unbounded closed set.

\begin{cor} \label{cor-X-one-pt}
Assume that   $X$ is a proper connected metric space
equipped with a locally doubling measure $\mu$ that
supports a local \p-Poincar\'e inequality, where $1<p<\infty$.
Let $Q \ge 1$.

If $X \setm \{x_0\}$  is sequentially
annularly quasiconvex at $x_0$, then
$\{x_0\}$ is removable for bounded 
$Q$-quasiharmonic
functions in $X \setm \{x_0\}$
if and only if at least one of the following conditions holds\/\textup:
\begin{enumerate}
\item
$X$ is  \p-parabolic or bounded,
\item 
$\Cp(\{x_0\})=0$.
\end{enumerate}
\end{cor}

Note that  the sequential
annular quasiconvexity in Corollary~\ref{cor-X-one-pt}
cannot be replaced by local connectedness, see
Example~\ref{ex-factorial2}.

Let $X$ be a connected $n$-dimensional Riemannian (or Finsler) manifold 
with or without a boundary,
i.e.\ $X=M \cup \bdy M$, where the manifold boundary
$\bdy M =\emptyset$ if $M$ is complete.
Then the canonical volume form $dV$ (= the normalized $n$-dimensional
Hausdorff measure) is locally doubling and 
supports a local \p-Poincar\'e inequality on $X$.
More generally, if we let $d\mu= w\,dV$, with a 
weight $0< w \in C(X)$, then $\mu$ is also 
locally doubling and 
supports a local \p-Poincar\'e inequality on $X$.
(If $0<w \in C^\infty(X)$, then the Riemannian manifold $X$, 
equipped with $\mu$,
is a so-called \emph{smooth metric measure space}.)
Hence, with the local assumptions in this paper, 
connected Riemannian (and Finsler) manifolds with  boundary and
equipped with
such measures are covered by our results,
provided that $X=M \cup \bdy M$ is proper.
Since Riemannian (and Finsler) manifolds (of dimension $n \ge 2$)
are sequentially
annularly quasiconvex we have the following result,
including the measures $dV$ and $\mu$ mentioned above.
The manifolds are not required to be oriented.

\begin{cor} \label{cor-mfld}
Let $X$ be a connected Riemannian\/ \textup(or Finsler\/\textup)
manifold with or without
boundary.  
Assume that $\dim X \ge 2$ and that $X$
is proper and
equipped with a locally doubling measure $\mu$ that
supports a local \p-Poincar\'e inequality.
Let $\Om \subset X$ be a nonempty open set and $K \subsetneq \Om$ be compact.

Then $K$ is removable for bounded $Q$-quasiharmonic functions in $\Om \setm K$
if and only if  at least one of the following conditions holds\/\textup:
\begin{enumerate}
\item $\Cp(K)=0$,
\item $X$ is \p-parabolic or bounded, and
$\Cp(K \setm \{x_0\})=\Cp(X\setm \Om)=0$ for some $x_0 \in K$.
\end{enumerate}
\end{cor}

When $X=\Rn$ (and $n \ge 2$) equipped with a globally doubling weight supporting
a global \p-Poincar\'e inequality (i.e.\ with a \p-admissible weight
as in Heinonen--Kilpel\"ainen--Martio~\cite{HeKiMa}), 
Corollary~\ref{cor-mfld}
is a special case of Bj\"orn~\cite[Theorem~1.3]{ABRnremove}. 
Here we only have local assumptions.

The following is another direct consequence of 
Theorem~\ref{thm-main-char-intro} (and removability of
sets with zero capacity).

\begin{cor} \label{cor-remove-K-from-X-hyp}
Assume that   $X$ is a proper connected metric space
equipped with a locally doubling measure $\mu$ that
supports a local \p-Poincar\'e inequality, where $1<p<\infty$.

Let $\Om \subset X$ be open, 
 $K \subsetneq \Om$ be compact and $Q\ge1$.
Also assume that at least one of the following conditions holds\/\textup:
\begin{enumerate}
\renewcommand{\theenumi}{\textup{(\arabic{enumi})}}%
\item 
$X$ is \p-hyperbolic,
\item
$\Cp(X\setm \Om)>0$,
\item
$\Cp(\{x\})=0$ for every $x \in \bdy K$.
\end{enumerate}
Then $K$ is removable for bounded $Q$-quasiharmonic functions in $\Om \setm K$
if and only if $\Cp(K)=0$.
\end{cor}

Examples~\ref{ex-hyp-R} and~\ref{ex-two-R} show that
Corollary~\ref{cor-remove-K-from-X-hyp}
fails if $K\subsetneq \Om$ is merely assumed to be closed (in $X$).

In \eqref{eq-seq=>lim=>loc-conn} and Example~\ref{ex-circle-new}
we saw that local connectedness is an important property in
connection with removability.
The following result shows that it is in fact
a necessary requirement, even for 
removability of relatively closed subsets.

\begin{prop} \label{prop-not-locconn}
Assume that   $X$ is a proper connected metric space
equipped with a locally doubling measure $\mu$ that
supports a local \p-Poincar\'e inequality, where $1<p<\infty$.

Let $E$ be a relatively closed subset of $\Om$.
Assume that there is $x_0 \in \Om \cap \bdy E$ 
such that $G:=\Om\setm E$ is not locally connected at $x_0$. 
Then there is a bounded \p-harmonic function in $G$
which does not have any continuous extension to $\Om$.

In particular, $E$ is not removable for bounded 
$Q$-quasiharmonic functions in $G$, when $Q \ge 1$.
\end{prop}

Along the way we obtain the following Liouville theorem
for nonnegative entire quasiharmonic functions
in \p-parabolic spaces.
In fact, it can be proved even for
quasisuperharmonic functions.
For bounded $X$, it follows directly from the strong minimum principle.

\begin{thm} \label{thm-Liouville-para}
\textup{(The Liouville theorem in \p-parabolic spaces)}
Assume that   $X$ is a proper connected metric space
equipped with a locally doubling measure $\mu$ that
supports a local \p-Poincar\'e inequality, where $1<p<\infty$.

If $X$ is \p-parabolic or bounded, then 
every nonnegative quasisuperharmonic function in $X$
is constant.
\end{thm}

Under global assumptions, the Liouville theorem
is well known and holds
also in \p-hyperbolic spaces, see 
Kinnunen--Shan\-mu\-ga\-lin\-gam~\cite{KiSh01} and
Kinnunen--Martio~\cite{KiMa03}
(and Theorem~\ref{thm-Liouville-global}).
Their proofs rely on the weak Harnack inequality.
Under \emph{local} assumptions, the Liouville theorem
does not follow from the weak Harnack inequality in the same way as it 
does under global assumptions.
In fact,  in this case, the Liouville theorem can fail 
for bounded \p-harmonic functions in \p-hyperbolic spaces,  
see Bj\"orn--Bj\"orn~\cite[Example~10.3]{BBsemilocal}.
Moreover, the Liouville theorem always fails for \emph{superharmonic} functions
in \p-hyperbolic spaces, because there 
are positive global Green functions 
(and they are superharmonic) in \p-hyperbolic spaces, 
while they never exist in \p-parabolic spaces, cf.\ Remark~\ref{rmk-hypend}.

For \p-harmonic functions, Theorem~\ref{thm-Liouville-para}
was obtained in 
Bj\"orn--Bj\"orn--Shan\-mu\-ga\-lin\-gam~\cite[Proposition~6.7]{BBShypend},
using the strong maximum principle (and thus relying 
on the weak Harnack inequality).
In contrast, our proof of Theorem~\ref{thm-Liouville-para}
is completely elementary and only relies on the
definition of \p-parabolicity.

\begin{remark} \label{rmk-hypend}
In the terminology of \cite{BBShypend},
Theorem~\ref{thm-Liouville-para} implies that 
$O^p_{\para}  \subset    O^p_{QP}$,
which improves upon Theorem~1.2 in~\cite{BBShypend}.
In the mean time it has also been
shown that $O^p_{\para}= O^p_{G}$, i.e.\ that $X$ is
\p-parabolic if and only if it does not carry a global positive Green function, see
\cite[Theorem~1.1]{BBglobal}.
For Riemannian manifolds the equality $O^p_{\para}= O^p_{G}$ 
is due to Holopainen~\cite[Theorem~3.27]{Ho}.
\end{remark}

In Section~\ref{sect-Liouville-G} we  use Theorem~\ref{thm-Liouville-para}
to give the following complete characterization of the Liouville theorem
for bounded  $Q$-quasiharmonic
functions in open sets $G$ with (possibly unbounded)
complement $F:=X \setm G$  
and $\Cp(F)>0$.
The case when $\Cp(F)=0$ is covered by Proposition~\ref{prop-L-CpF=0-new}.
Note that
if  the Liouville theorem holds in $G$, then $G$ must be connected.

\begin{thm} \label{thm-LL-F}
Assume that   $X$ is a proper connected metric space
equipped with a locally doubling measure $\mu$ that
supports a local \p-Poincar\'e inequality, where $1<p<\infty$.

Let $G \subset X$ be  a nonempty open set such that
$\Cp(X \setm G)>0$, and let $Q \ge 1$.
Then the Liouville theorem holds for bounded
$Q$-quasiharmonic functions in $G$ 
if and only if all of the following conditions hold\textup:
\begin{enumerate}
\renewcommand{\theenumi}{\textup{(\roman{enumi})}}%
\item  \label{Lio-par}
$G$ is bounded or a \p-parabolic set\/ 
\textup(see~\eqref{eq-par-set-int}\textup),
\item \label{Lio-cap}
there is $x_0 \in \bdy G$ such that $\Cp(\bdy G \setm \{x_0\})=0$,
\item \label{Lio-lim}
the limit 
\begin{equation*} 
\lim_{G \ni y \to x_0} u(y) \quad \text{exists}
\end{equation*}
for every bounded $Q$-quasiharmonic function $u$ in~$G$.
\end{enumerate}
In particular, with $Q=1$, the statement holds for \p-harmonic functions.
\end{thm}

The Liouville theorem trivially implies removability
of $F:=X \setm G$, but 
Examples~\ref{ex-hyp-R} and~\ref{ex-two-R} show
that the converse implication does not hold if $F$ is not compact,
in contrast to Theorem~\ref{thm-main-char-intro}.

Removable singularities for bounded harmonic functions is a classical topic
in potential theory.
Let us just mention that Bouligand~\cite[p.~104]{Bouligand26} showed that 
compact sets 
with zero capacity are removable for bounded
harmonic functions in open subsets of $\Rn$.
Removability of sets with zero capacity 
for bounded \p-harmonic functions in open subsets of $\Rn$
is due to Serrin~\cite[Theorem~10]{serrin64} (compact sets in unweighted $\Rn$)
and
Hei\-no\-nen--Kil\-pe\-l\"ai\-nen--Martio~\cite[Theorem~7.36]{HeKiMa} (relatively closed sets
in weighted $\Rn$ with a \p-admissible weight).
On metric spaces this  was obtained in
Bj\"orn--Bj\"orn--Shan\-mu\-ga\-lin\-gam~\cite[Proposition~8.3]{BBS2}
and (a bit more generally and also for quasiharmonic functions) 
Bj\"orn~\cite[Theorem~6.2]{ABremove}.
Removability of compact sets with
zero capacity for bounded quasiharmonic 
functions in open subsets of unweighted $\Rn$ was shown
by Tolksdorf~\cite[Theorem~1.5]{tolksdorf}.
Kilpel\"ainen~\cite[Theorem~1.8]{Kilp89} showed 
that the Liouville theorem holds for bounded \p-harmonic functions
in a connected open set $G$ in unweighted $\Rn$, with $1<p\le n$, if and only if
$\Cp(\Rn \setm G)=0$.
We saw in Example~\ref{ex-interval-new}(b)
that compact sets with positive capacity can be removable.
That there are such exceptional cases also on unweighted (and weighted) $\Rn$, 
$n \ge 2$, was shown in Bj\"orn~\cite[Theorems~1.2--1.3] {ABRnremove}.

The outline of the paper is as follows:
In Sections~\ref{sect-ug} and~\ref{sect-harm},
we introduce the necessary background  from
first-order analysis on metric spaces.

The Dirichlet problem
is a useful tool for constructing \p-harmonic functions
with various properties.
In this paper we will use two types of solutions
of the Dirichlet problem: Sobolev and Perron solutions.
They are defined in Section~\ref{sect-Perron}.

The main result (Theorem~\ref{thm-cpt-pharm-nec}) in Section~\ref{sect-cpt}
shows that a number of specific properties has to hold if a compact
set of positive capacity should be removable. 

In Section~\ref{sect-Liouville} we discuss a different version
of the Liouville theorem on $X$, and give a proof
of it for quasisuperharmonic functions on \p-parabolic spaces 
(Theorem~\ref{thm-Liouville-para}).
Also Theorem~\ref{thm-main-char-intro} 
and Corollary~\ref{cor-remove-K-from-X-hyp} are proved here.
In Section~\ref{sect-Liouville-G} we turn to
the Liouville theorem in open subsets and prove Theorem~\ref{thm-LL-F}.

Section~\ref{sect-loc-conn} is devoted to 
local connectedness and 
Proposition~\ref{prop-not-locconn},
while Section~\ref{sect-annular} deals with
sequential annular quasiconvexity,
leading to the proofs of Corollaries~\ref{cor-X-one-pt} and~\ref{cor-mfld}.
We end the paper with various examples in Sections~\ref{sect-examples}
and~\ref{sect-Janas-ex}.
In particular,  Example~\ref{ex-two-R}
provides unbounded closed removable sets that separate the space
and Section~\ref{sect-Janas-ex} contains examples of locally connected spaces 
without annular quasiconvexity.

\section{Notation and preliminaries}
\label{sect-ug}

\begin{ass}
We assume throughout the paper
that $X=(X,d,\mu)$ is a metric space equipped
with a metric $d$ and a positive complete  Borel  measure $\mu$
such that $0<\mu(B)<\infty$ for all 
balls $B \subset X$.
We also assume that $1<p< \infty$
and that $\Om\ne\emptyset$ is an open set in $X$.
Additional general
assumptions are added at  the beginning of Section~\ref{sect-harm}.
\end{ass}

In this section we will introduce and give precise definitions
of the necessary metric space concepts used in this paper.
We will be brief, 
for further details
see the monographs Bj\"orn--Bj\"orn~\cite{BBbook} and
Heinonen--Koskela--Shan\-mu\-ga\-lin\-gam--Tyson~\cite{HKST},
where the theory is thoroughly  developed with proofs.

It follows from our general assumptions that $X$ is separable and Lindel\"of.
To avoid pathological situations we assume that $X$ contains
at least two points.

A \emph{curve} is a continuous mapping from an interval,
and a \emph{rectifiable} curve is a curve with finite length.
A rectifiable curve can
be parameterized by its arc length $ds$.
A property holds for \emph{\p-almost every nonconstant rectifiable curve}
if the curve family $\Ga$ for which it fails has zero \p-modulus,
i.e.\ there is $\rho\in L^p(X)$ such that
$\int_\ga \rho\,ds=\infty$ for every $\ga\in\Ga$.
We will only consider curves which are 
compact and rectifiable.
Following Heinonen--Koskela~\cite{HeKo98} and Koskela--MacManus~\cite{KoMc}
we next introduce upper gradients and \p-weak upper gradients.

\begin{deff} \label{deff-ug}
A Borel function $g:X \to [0,\infty]$ is an \emph{upper gradient}
of $f:X \to \eR:=[-\infty,\infty]$
if for every nonconstant rectifiable curve
$\gamma: [0,\ell_{\gamma}] \to X$,
\begin{equation} \label{ug-cond}
        |f(\gamma(0)) - f(\gamma(\ell_{\gamma}))| \le \int_{\gamma} g\,ds,
\end{equation}
where the left-hand side is $\infty$
whenever at least one of the
terms therein is infinite. 

If $g:X \to [0,\infty]$ is  measurable 
and \eqref{ug-cond} holds for \p-almost every 
nonconstant rectifiable curve,
then $g$ is a \emph{\p-weak upper gradient} of~$f$.
\end{deff}

If $g \in \Lploc(X)$ is a \p-weak upper gradient of $f$,
then one can find a sequence $\{g_j\}_{j=1}^\infty$
of upper gradients of $f$ such that
$\lim_{j \to \infty}\|g_j-g\|_{L^p(X)} =0$.
If $f$ has an upper gradient in $\Lploc(X)$, then
it has a \emph{minimal \p-weak upper gradient}
$g_f \in \Lploc(X)$ in the sense that $g_f \le g$ a.e.\ 
for every \p-weak upper gradient $g \in \Lploc(X)$ of $f$.

Following Shan\-mu\-ga\-lin\-gam~\cite{Sh-rev}, 
we next define the Newtonian Sobolev space on the metric space $X$.

\begin{deff} \label{deff-Np}
For measurable $f$, let
\[
        \|f\|_{\Np(X)} = \biggl( \int_X |f|^p \, d\mu
                + \inf_g  \int_X g^p \, d\mu \biggr)^{1/p},
\]
where the infimum is taken over all upper gradients of $f$
(or equivalently over all \p-weak upper gradients of $f$).
The \emph{Newtonian space} on $X$ is
\[
        \Np (X) = \{f: \|f\|_{\Np(X)} <\infty \}.
\]
\end{deff}

The space $\Np(X)/{\sim}$, 
with the equivalence relation
$f \sim h$ if and only if $\|f-h\|_{\Np(X)}=0$,
is a Banach space and a lattice.
The \emph{Dirichlet space} $\Dp(X)$ is the collection of all
measurable functions on $X$ 
that have a \p-weak upper gradient in $L^p(X)$.

We say  that $f \in \Nploc(X)$ if
for every $x \in X$ there is a ball $B_x\ni x$ such that
$f \in \Np(B_x)$. 
If $f,h \in \Nploc(X)$,
then $g_f=g_h$ a.e.\ in $\{x \in X : f(x)=h(x)\}$.
In particular, $g_{\min\{f,c\}}=g_f \chi_{\{f < c\}}$ a.e.\ in $X$
for $c \in \R$.
For a measurable set $E\subset X$, the spaces 
$\Np(E)$, $\Nploc(E)$ and $\Dp(E)$
are defined by
considering $(E,d|_E,\mu|_E)$ as a metric space in its own right.
In this paper, it is convenient to
assume that functions in $\Np$ and $\Dp$
 are defined everywhere (with values in $\eR$),
not just up to an equivalence class in the corresponding function space.
This is e.g.\ essential for the definitions of
upper gradients and \p-weak upper gradients to make sense.

The (Sobolev) \emph{capacity} of an arbitrary set $E\subset X$ is 
\[
\Cp(E) = \inf_u\|u\|_{\Np(X)}^p,
\]
where the infimum is taken over all $u \in \Np(X)$ such that
$u\geq 1$ in $E$.
A property holds \emph{quasieverywhere} (q.e.)\
if the set of points  for which it fails
has capacity zero.
The capacity is the correct gauge
for distinguishing between two Newtonian functions:
If $u \in \Np(X)$, then 
$v \sim u$ if and only if $v=u$ q.e.
Moreover, if $u,v \in \Np(X)$ and $u= v$ a.e., then $u=v$ q.e.

For a ball $B=B(x,r):=\{y : d(x,y)<r\}$ with centre $x$ and radius $r$, we let
$\lambda B = B(x, \lambda r)$. In metric spaces
it can happen that balls with different centres or
radii denote the same set.
We will, however,
make the convention that a ball $B$ comes with a predetermined
centre and radius.
All balls are assumed to be open in this paper.

A measure $\mu$ is \emph{Ahlfors $D$-regular} if there
is a constant $C \ge 1$ such that
\[
     \frac{r^D}{C} \le \mu(B(x,r)) \le Cr^D
\quad \text{for every $x \in X$ and $0 <r< 2 \diam X$}.
\]

\begin{deff} 
The measure $\mu$ is \emph{doubling within  $\Om$}
if there is a constant $C>0$  
such that $\mu(2B)\le C \mu(B)$ for all balls $B \subset \Om$.

Similarly, the
\emph{\p-Poincar\'e inequality holds within $\Om$}, with $1 \le p <\infty$,
if there is a constant $C>0$ and a \emph{dilation constant} $\lambda \ge 1$
such that for all balls $B\subset \Om$,
all integrable functions $u$ in $\la B$ and all 
\p-weak upper gradients $g$ of $u$ in $\la B$,
\begin{equation*} 
        \vint_{B} |u-u_B| \,d\mu
        \le C r_B \biggl( \vint_{\lambda B} g^{p} \,d\mu \biggr)^{1/p},
\end{equation*}
where $u_B:=\vint_B u \,d\mu := \int_B u\, d\mu/\mu(B)$
and $r_B$ is the radius of $B$.

Each of these properties is called \emph{local} if 
for every $x \in X$ there is $r>0$
such that the property holds within $B(x,r)$.
If a property holds within $\Om=X$, then it is called \emph{global}.
\end{deff}

In $\R^n$ equipped with a globally doubling measure $d\mu=w\,dx$, 
the global \p-Poincar\'e inequality 
is equivalent to the \emph{\p-admissibility} of the weight $w$ in the
sense of Hei\-no\-nen--Kil\-pe\-l\"ai\-nen--Martio~\cite{HeKiMa}, see
Corollary~20.9 in~\cite{HeKiMa}
and Proposition~A.17 in~\cite{BBbook}.
Moreover, in this case $g_u=|\nabla u|$ if $u \in \Np(\R^n)$.

The space $X$ is \emph{proper} if  
all closed bounded subsets are compact.
As usual, we write $f_+= \max\{f,0\}$.
In this  paper, a continuous function is always assumed
to be real-valued (as opposed to $\eR$-valued). 

An \emph{$L$-quasiconvex curve} between two points $x$ and $y$ is
a curve with length $\ell \le L d(x,y)$. 
The space $X$ is \emph{$L$-quasiconvex} if every pair of points
can be connected by an $L$-quasiconvex curve.
It is \emph{geodesic} if every pair of points
can be connected by a \emph{geodesic}, i.e.\ a  $1$-quasiconvex curve.
$X$ is \emph{quasiconvex} if it is $L$-quasiconvex for some $L$.

\section{\texorpdfstring{\p}{p}-harmonic and quasiharmonic functions}
\label{sect-harm}

\begin{ass}
In addition to the assumptions from the beginning of Section~\ref{sect-ug},
  we assume from now on that
  $X$ is a proper connected metric space
equipped with a locally doubling measure $\mu$ that
supports a local \p-Poincar\'e inequality, where $1<p<\infty$.
We will also always assume that $1 \le Q < \infty$.
\end{ass}

By Proposition~1.2 and Theorem~1.3 in
Bj\"orn--Bj\"orn~\cite{BBsemilocal},
it follows from our assumptions
that $\mu$ is doubling and supports a \p-Poincar\'e
inequality within every ball.
(These properties were  called semilocal in~\cite{BBsemilocal}.)
It also follows that $X$ is locally quasiconvex, see
\cite[Proposition~4.8]{BBsemilocal},
which in turn implies that $X$ is locally connected 
(see Definition~\ref{deff-lc}),
and hence components of open sets are open.
Moreover, components of open sets are rectifiably connected,
see Lemma~4.38 in~\cite{BBbook}.

In this section we recall the definitions of
\p-harmonic, quasiharmonic and superharmonic functions
and present some of their important properties that
will be needed later.
The dependence on $p$ is implicit in the notation,
but $1<p<\infty$ is always fixed.

Even though we do not assume global doubling and a global \p-Poincar\'e inequality,
all results
in Chapters~6--14 in~\cite{BBbook} (except for the Liouville theorem)
hold under these assumptions,
as well as the results in 
Bj\"orn--Hansevi~\cite{BHansevi1} and
Hansevi~\cite{Hansevi1}, \cite{Hansevi2}, 
see the discussions in~\cite[Section~10]{BBsemilocal}
and Bj\"orn--Bj\"orn--Shan\-mu\-ga\-lin\-gam~\cite[Remark~3.7]{BBShypend}.
Also all earlier quantitative results that we cite about
\p-harmonic functions and quasiminimizers  
(except for the Liouville theorem)
hold under local assumptions.

\begin{deff} \label{def-quasimin}
A function $u \in \Nploc(\Om)$ is a
\emph{$Q$-quasiminimizer} in $\Om$ if 
\[ 
      \int_{\phi \ne 0} g^p_u \, d\mu
           \le Q \int_{\phi \ne 0} g_{u+\phi}^p \, d\mu
           \quad \text{for all } \phi \in \Np_0(\Om),
\] 
where
\[
\Np_0(\Om):=
  \{\phi|_{\Om} : \phi \in \Np(X) \text{ and }
        \phi=0 \text{ on } X \setm \Om\}.
\]
A \emph{$Q$-quasiharmonic function} 
is a continuous $Q$-quasiminimizer,
and a \emph{\p-harmonic function} is a $1$-quasiharmonic function.
A function is \emph{quasiharmonic} if it is
$Q$-quasiharmonic for some $Q$.
\end{deff}

For various characterizations of quasiminimizers
see Bj\"orn~\cite{ABkellogg}.
It was shown in Kinnunen--Shan\-mu\-ga\-lin\-gam~\cite{KiSh01} 
(see also \cite[Section~10]{BBsemilocal}) that 
a $Q$-quasiminimizer 
can be modified on a set of zero (Sobolev) capacity to obtain
a $Q$-quasiharmonic function,
and that this function is locally H\"older continuous.
Moreover, they obtained Harnack inequalities,
the following strong maximum principle and the Liouville theorem
(Theorem~\ref{thm-Liouville-global}).

\begin{thm} \label{thm-max}
\textup{(The strong maximum principle \cite[Corollary~6.4]{KiSh01})}
  A quasiharmonic function that attains
  its maximum in a connected open set $\Om$ must be constant in $\Om$.
\end{thm}

\begin{deff} \label{deff-superharm-class}
A function $u : \Om \to (-\infty,\infty]$ is 
\emph{superharmonic} in $\Om$ if
\begin{enumerate}
\renewcommand{\theenumi}{\textup{(\roman{enumi})}}%
\item \label{cond-a} $u$ is lower semicontinuous,
\item \label{cond-b} 
 $u$ is not identically $\infty$ in any component of $\Om$,
\item \label{cond-c}
for every nonempty open set $V \Subset \Om$ with $\Cp(X \setm V)>0$
and every function $v\in C(\clV)$ that
is \p-harmonic in $V$ and such
that $v\le u$ on $\bdy V$, we have $v\le u$ in $V$.
\end{enumerate}
\end{deff}

(As usual, by $V \Subset \Om$ we mean that $\clV$
is a compact subset of $\Om$.)

This definition of superharmonicity is the same as the one
usually used in the Euclidean literature, e.g.\ in
Hei\-no\-nen--Kil\-pe\-l\"ai\-nen--Martio~\cite[Section~7]{HeKiMa}.
It is equivalent to other definitions of superharmonicity on metric spaces,
by Theorem~6.1 in Bj\"orn~\cite{ABsuper}
(or \cite[Theorem~14.10]{BBbook}).
It is not difficult to see that a function $u$ is \p-harmonic if and only if 
both $u$ and $-u$ are superharmonic.

The following result shows
that sets with zero capacity are always removable. 
This is the reason why 
we concentrate 
on removable sets with positive capacity.

\begin{thm} \label{thm-removability-qharm}
\textup{(Bj\"orn~\cite[Theorem~6.2]{ABremove}
(or \cite[Theorem~12.2]{BBbook}))}
Let $E \subset \Om$ be a relatively closed subset such that $\Cp(E)=0$.
Then every bounded $Q$-quasiharmonic function $u$ in $\Om \setm E$
has a unique bounded  $Q$-quasiharmonic extension to $\Om$.
\end{thm}

The classification of unbounded metric spaces as parabolic and
hyperbolic is important in this paper.
For this we first need to define the condenser capacity 
$\cp(K,\Om)$, which we only need for compact $K \subset \Om$.
In this case, it can be defined using Lipschitz functions
or, equivalently, by using $\Np_0(\Om)$-functions,
see e.g.\ Kallunki [Rogovin]--Shan\-mu\-ga\-lin\-gam~\cite{KaSh} or
\cite[Definition~3.1 and Proposition~7.8]{BBglobal}.

\begin{deff} \label{deff-cp}
Let $\Om\subset X$ be a (possibly unbounded) open set.
The \emph{condenser capacity} of a compact set $K \subset \Om$
with respect to $\Om$ is
\begin{equation*} 
\cp(K,\Om) = \inf_{\substack{u \in \Lipc(\Om) \\u \ge 1 \text{ on } K}}
\int_{\Om} g_u^p\, d\mu,
\end{equation*}
where $\Lipc(\Om)=\{f \in \Lip(X) : \supp f \Subset \Om \}$.
\end{deff}

\begin{deff} \label{def-p-par}
Assume that $X$ is unbounded.
Then $X$ is called \emph{\p-hyperbolic}
if $\cp(K,X)>0$ for some compact set $K\subset X$.
Otherwise, $X$ is \emph{\p-parabolic}. 
\end{deff}

The following 
characterization of \p-parabolicity (under global assumptions) was obtained 
in Bj\"orn--Bj\"orn--Lehrb\"ack~\cite{BBLintgreen}.
(Note that Theorem~1.0.1 in Keith--Zhong~\cite{KeithZhong}
provides us with the better Poincar\'e inequality required in~\cite{BBLintgreen}.)

\begin{thm}   \label{thm-p-parab}
\textup{(Theorem~5.5 in~\cite{BBLintgreen})}
Assume that $X$ is unbounded and  $x_0 \in X$. 
If 
\begin{equation}   \label{eq-int-parab=infty}
\int_{1}^\infty \biggl( \frac{r}{\mu(B(x_0,r))} \biggr)^{1/(p-1)} \,dr=\infty,
\end{equation}
then $X$ is \p-parabolic.

If $\mu$ is globally doubling and supports a global \p-Poincar\'e
inequality, then 
$X$ is \p-parabolic 
if and only if \eqref{eq-int-parab=infty} holds.
\end{thm}

In particular, if $\mu$ is Ahlfors \p-regular
 and supports a global \p-Poincar\'e inequality,
then $X$ is \p-parabolic.

We will need the following two topological lemmas. 
The proof of Lemma~\ref{lem-Cp-connected}, given in~\cite{BBbook},
only uses the \p-Poincar\'e inequality on each ball separately, without
assuming any
uniformity in the constants, and is thus available here.

\begin{lem} \label{lem-Cp-connected}
\textup{(\cite[Lemmas~4.5 and 4.6]{BBbook})}
Assume that\/ $\Om$ is connected
and let $E \subsetneq \Om$ be relatively closed.
\begin{enumerate}
\item \label{it-Cp} 
$\Cp(E)=0$ if and only if $\Cp(\bdy E \cap \Om)=0$.
\item \label{it-conn}
If $\Cp(E)=0$, then $\Om \setm E$ is connected.
\end{enumerate}
\end{lem}

\begin{lem} \label{lem-G_0-open}
Let $G$ be open 
and assume that 
$\Cp(\bdy G \setm \{x_0\})=0$
for some $x_0 \in \bdy G$.
Then the set $G_0:=\clG \setm \{x_0\}$ is open.
\end{lem}

\begin{proof}
It suffices to consider $y\in \bdy G\setm\{x_0\}$ and show that 
a neighbourhood of $y$ is contained in $G_0$.
Choose 
$0<\de<d(x_0,y)$
and let $V$ be the component of $B(y,\de)$ that contains $y$. 
(Recall that balls need not be connected.)
Then $B(y,\de')\subset V$ for some $\de'>0$ 
(because $X$ is locally connected).
Since $y\in \bdy G$,
there exists $z\in B(y,\de')\cap G  \subset V\cap G$.

Because $\Cp(\bdy G \setm \{x_0\})=0$
and $x_0 \notin V$, Lemma~\ref{lem-Cp-connected}\ref{it-conn}
implies that the set $V\setm \bdy G$ is connected.
At the same time,  $V\setm \bdy G$ can be written as a disjoint union of the
open sets $V\cap G\ni z$ and $V\setm \clG$.
Hence $V\setm \clG=\emptyset$
and consequently $V\subset \clG$.
As $x_0 \notin V$, we conclude that $V \subset G_0$,
so $G_0$ is indeed open.
\end{proof}

\section{Sobolev and Perron solutions of the Dirichlet problem}
\label{sect-Perron}

\emph{Recall the general assumptions from the beginning of Sections~\ref{sect-ug}
and~\ref{sect-harm}.}

\medskip

The Dirichlet problem is  
a useful tool for constructing \p-harmonic functions
with various properties.
We will use two types of solutions of 
the Dirichlet problem:
Sobolev and Perron solutions.
(The Sobolev solutions are 
sometimes called \p-harmonic extensions,
but they are not extensions in the sense used in this paper,
so we will avoid that terminology.) 
Following Hansevi~\cite[Definition~4.6]{Hansevi1}
we define the Sobolev solutions for boundary data in the
Dirichlet space $\Dp(\Om)$  as follows.
They will be used  to prove Proposition~\ref{prop-not-locconn}
in Section~\ref{sect-loc-conn}.

\begin{deff}  \label{def-harm-ext}
Assume that $\Om$ is an open set with $\Cp(X \setm \Om)>0$.
Let $f \in \Dp(\Om)$.
Then the \emph{Sobolev solution}
$H_\Om f$ of the Dirichlet problem with boundary data
$f$ in $\Om$ is the unique continuous function
in $\Om$ such that $f-H_\Om f \in \Dp_0(\Om)$ and
\[
     \int_\Om g_{H_\Om f}^p \,d\mu 
    \le \int_\Om g_{H_\Om f+\phi}^p \,d\mu \quad
\text{for every } \phi \in \Dp_0(\Om),
\]
where
\[
\Dp_0(\Om):=
  \{\phi|_{\Om} : \phi \in \Dp(X) \text{ and }
        \phi=0 \text{ on } X \setm \Om\}.
\]
\end{deff}

The Sobolev solution $H_\Om f$ exists  
and is unique by Hansevi~\cite[Theorem~4.4]{Hansevi1}.
For bounded $\Om$ and $f \in \Np(X)$,
the definition
of $H_\Om f$ coincides with other definitions in the literature, such
as in Shan\-mu\-ga\-lin\-gam~\cite[Theorem~5.6]{Sh-harm},
Bj\"orn--Bj\"orn--Shan\-mu\-ga\-lin\-gam~\cite[Definition~3.3]{BBS}
and Bj\"orn--Bj\"orn~\cite[Definition~8.31]{BBbook}.
The existence, uniqueness and other properties of $H_\Om f$ in bounded
sets were obtained in these references.

Perron solutions make it possible to solve the Dirichlet problem
for general boundary data.
They 
will always be considered with respect to the
extended boundary 
\[
\bdystar \Om:= \begin{cases}   
     \bdy\Om \cup \{\binfty\},   &  \text{if $\Om$ is unbounded,}  \\ 
     \bdy \Om,  & \text{otherwise,}  
\end{cases}
\]
where $\binfty$ is the point added in the \emph{one-point compactification} 
$X \cup \{\binfty\}$ of $X$ when $X$ is unbounded.
In particular, $\bdystar \Om$ is always compact.
Note that $\bdystar \Om=\emptyset$ if and only if $X$ is bounded and $\Om=X$.
In this case, the Dirichlet problem for \p-harmonic functions
does not make sense.

\begin{deff}\label{def:Perron}
Assume that $\bdystar \Om \ne \emptyset$.
  Given $f:\bdystar\Omega\to\eR$, 
let $\UU_f(\Omega)$ be the collection of all superharmonic functions 
$u$ in $\Omega$ that are bounded from below and such that 
\[
	\liminf_{\Omega\ni y\to x} u(y) \geq f(x)
	\quad\textup{for all }x\in\bdystar\Omega.
\]
The \emph{upper Perron solution} of $f$ is defined by 
\[
	\uP_\Omega f(x)
	= \inf_{u\in\UU_f(\Omega)} u(x),
	\quad x\in\Omega.
\]
The lower Perron solution is defined by
$\lP_\Omega f = - \uP_\Omega f$.
When $\uP_\Omega f=\lP_\Omega f$, we 
let $P_\Omega f=\uP_\Omega f$.
If in addition $P_\Omega f$ is real-valued, then $f$ is said to be 
\emph{resolutive} (for~$\Omega$). 
\end{deff}

In each component of $\Omega$, $\uP_\Om f$ is either \p-harmonic or 
identically $\pm\infty$, 
by Theorem~4.1 in Bj\"orn--Bj\"orn--Shan\-mu\-ga\-lin\-gam~\cite{BBS2}
(or \cite[Theorem~10.10]{BBbook}) whose
proof applies also to unbounded $\Om$.
Moreover, it follows from the comparison principle
between sub- and superharmonic functions in 
Hansevi~\cite[Theorem~6.2]{Hansevi2} that
\begin{equation*} 
  \lP_\Om f \le \uP_\Om f
  \quad \text{for all } f: \bdystar \Om \to \eR.
\end{equation*}  
(In~\cite{Hansevi2} it is assumed that $\Cp(X \setm\Om)>0$,
but this is not needed in the proof, provided that
$\bdystar \Om \ne \emptyset$.)

Under the assumptions in this paper, 
continuous functions on $\bdystar \Om$ need not be resolutive,
see e.g.\ Bj\"orn--Bj\"orn~\cite[Proposition~11.2]{BBglobal}.
We therefore make the following definition.

\begin{deff}\label{def:reg}
A boundary point $x_0\in\bdystar\Omega$ is \emph{regular} 
(for $\Om$) 
if 
\[
	\lim_{\Omega\ni y\to x_0}\uP_\Om f(y)
	= f(x_0)
	\quad\text{for all }f\in C(\bdystar\Omega).
\]
A boundary point $x_0 \in \bdystar\Omega$ is \emph{irregular} if
it is not regular.
\end{deff}

\begin{thm} \label{thm-Kellogg}
\textup{(The Kellogg property, Bj\"orn--Hansevi~\cite[Theorem~1.2]{BHansevi1})}
If $I_\Om \subset \bdystar \Om$ is the set of irregular boundary
points, then $\Cp(I_\Om \setm \{\binfty\})=0$.
\end{thm}

For bounded $\Om$ the Kellogg property was earlier obtained by 
Bj\"orn--Bj\"orn--Shan\-mu\-ga\-lin\-gam~\cite[Theorem~3.9]{BBS}, 
\cite[Theorem~6.1]{BBS2}.

\section{Removability of compact sets} 
\label{sect-cpt}

\emph{Recall the general assumptions from the beginning of Sections~\ref{sect-ug}
and~\ref{sect-harm}.}

\medskip

We now consider the case when $K\subsetneq \Om$ is compact.
(We exclude the case $K=\Om$, which can only happen when $K=\Om=X$.
In this case it does not make   
 sense to talk about removability for functions defined in
$\Om \setm K=\emptyset$.)

The following result shows, in particular,
that a compact set with positive
capacity is never removable if $\Cp(X \setm \Om)>0$
or if $X$ is \p-hyperbolic.

\begin{thm} \label{thm-cpt-pharm-nec}
Assume that $K\subsetneq\Om$ is compact with $\Cp(K)>0$.
Also assume that every bounded 
\p-harmonic function in $G:=\Om \setm K$
has a quasiharmonic extension to $\Om$,
which in particular holds if 
$K$ is removable for bounded $Q$-quasiharmonic functions in $G$.

Then all of the following conditions are satisfied\/\textup:
\begin{enumerate}
\item \label{r-1}
$\Cp(X \setm \Om)=0$,
\item \label{r-a}
$\Om$ is connected,
\item \label{r-2}
$G$ is connected, 
\item \label{r-3b}
$\Cp(\bdy K)>0$,
\item \label{r-3}
 there is
$x_0 \in \bdy K$ such that $\Cp(\bdy K \setm \{x_0\})=0$,
\item \label{r-5}
$X$ is \p-parabolic or bounded.
\end{enumerate}
\end{thm}

From \ref{r-3b} and \ref{r-3} it follows that $\Cp(\{x_0\})>0$.
The conditions in 
Theorem~\ref{thm-cpt-pharm-nec} are alone not enough to guarantee
removability, as seen by 
Example~\ref{ex-factorial2} below.
(Depending on the choices in that construction
it yields a bounded or a \p-parabolic space.)
Removability is not even guaranteed
if we add the  requirement that 
$G$ is locally connected at every $x\in \bdy K$,
which is a necessary condition
by Proposition~\ref{prop-not-locconn}.
Examples~\ref{ex-hyp-R} and~\ref{ex-two-R} show that 
Theorem~\ref{thm-cpt-pharm-nec} fails if $K \subsetneq X$ is merely
assumed to be (a possibly unbounded) closed subset of $X$.

\begin{proof}
\ref{r-1}
Assume, for a contradiction, that $\Cp(X \setm \Om)>0$.
As $X$ is Lindel\"of, there is a component $\Om'$ 
of $\Om$ such that $\Cp(K \cap \Om')>0$.
As $\Cp(X \setm \Om') \ge \Cp(X \setm \Om)>0$,
it follows from Lemma~\ref{lem-Cp-connected}\ref{it-Cp} 
that $\Cp(\bdy \Om')>0$ and also that $\Cp(\bdy (K\cap \Om'))>0$.
By the Kellogg property (Theorem~\ref{thm-Kellogg}), 
there are at least two regular boundary points for $G$:
$x_0 \in \bdy \Om'$ and $x_1 \in \bdy (K \cap \Om')$.

Since $K\subset\Om$ is compact, we see that 
$f:=\chi_{K} \in C(\bdy^* G)$ 
(with $f(\binfty)=0$).
Then  $u:=\uP_{G} f$ is \p-harmonic in $G$
and $0 \le u \le 1$ in $G$.
By assumption,  $u$ has a 
quasiharmonic  (and thus continuous) extension $\ut$ to $\Om$.
Since $x_0$ and $x_1$ are regular, we see that 
\[
\lim_{G\ni y\to x_j} \ut(y) = \lim_{G\ni y\to x_j} u(y) = f(x_j)=j, 
\quad j=0,1.
\]
In particular, $\ut$ is nonconstant in $\Om'$ and
\[
   \max_{\Om'} \ut = \max_{K\cap\Om'} \ut 
  \ge \ut(x_1) =1 \ge \max_G u
\]
is attained at some point in the compact set $K \cap \Om'$.
But this contradicts the strong maximum principle
(Theorem~\ref{thm-max}).
Hence, $\Cp(X \setm \Om)=0$.

\ref{r-a}
Since $\Cp(X \setm\Om)=0$ (by \ref{r-1}), 
it follows from Lemma~\ref{lem-Cp-connected}\ref{it-Cp} 
(with $\Om$ replaced by $X$) that $\Om$ is connected.

\ref{r-2} 
Let $G'$ be a component of $G$ and $u=\chi_{G'}:G \to \R$,
which is a bounded \p-harmonic function in $G$.
By assumption, $u$ has a quasiharmonic (and thus continuous)
extension $\ut$ to $\Om$.
Then either
\[
\max_K \ut \ge 1 = \max_G u \quad \text{or} \quad
\max_K \ut <1 =  \max_{G} u.
\]
In both cases, $\max_{\Om}\ut$ 
is attained at some point in $\Om$
(in $K$ or $G'$, respectively).
Since $\Om$ is connected (by \ref{r-a}), it
follows from the strong maximum principle
(Theorem~\ref{thm-max}), that $\ut$ is constant in $\Om$.
Hence $G=G'$, i.e.\  $G$ is connected.

\ref{r-3b}
Since $\Om$ is connected (by \ref{r-a}) and $\Cp(K)>0$, 
this follows from Lemma~\ref{lem-Cp-connected}\ref{it-Cp}.

\ref{r-3} Assume, for a contradiction, that $\Cp(\bdy K \setm \{x\})>0$
for every $x\in \bdy K$.
By the Kellogg property (Theorem~\ref{thm-Kellogg}), 
there is a regular boundary point $x_0 \in \bdy K$ for $G$.
As $\Cp(\bdy K \setm \{x_0\})>0$ there is, again by the Kellogg property,
another regular boundary point
$x_1 \in \bdy K \setm \{x_0\}$  (for $G$).
We can now obtain a contradiction just as in the second paragraph of
the proof of \ref{r-1} above 
(with $\Om'=\Om$ since $\Om$ is connected by \ref{r-a}).
Hence \ref{r-3} must hold.

\ref{r-5} 
It follows from  the already proved properties \ref{r-3b} 
and \ref{r-3} that $\Cp(\{x_0\})>0$ for some $x_0\in \bdy G$.
If $X$ is \p-hyperbolic,
Theorem~1.1 in Bj\"orn--Bj\"orn~\cite{BBglobal}
provides us with a  Green function $u$ in $X$ with singularity at $x_0$.
In particular, $u$ is 
a positive bounded \p-harmonic function in $X \setm \{x_0\}$
such that
\[
0 < u(x_0) = \lim_{y\to x_0}u(y) = \max_{X} u < \infty
\quad \text{and} \quad
\inf_X u=0.
\]
(The boundedness of $u$ in $X$
follows from 
$\Cp(\{x_0\})>0$ and Lemma~9.5 in~\cite{BBglobal}.)
By the removability assumption for $K$, 
 $u|_G$ extends as a 
quasiharmonic (and thus continuous) function $\ut$ 
in $\Om$.  
By continuity, $\ut(x_0)=u(x_0)$, and thus the  maximum 
\[
   \max_\Om \ut = \max_K \ut \ge
\ut(x_0)
\]
is attained at some point in the compact set $K$.
But this contradicts the strong maximum principle
(Theorem~\ref{thm-max}).
Therefore  \ref{r-5} must hold.
\end{proof}

\section{The Liouville theorem on \texorpdfstring{$X$}{X}
and Theorems~\ref{thm-main-char-intro} and~\ref{thm-Liouville-para}}
\label{sect-Liouville}

\emph{Recall the general assumptions from the beginning of Sections~\ref{sect-ug}
and~\ref{sect-harm}.}

\medskip

In this section we will 
prove the Liouville theorem (Theorem~\ref{thm-Liouville-para})
for quasisuperharmonic functions in \p-parabolic spaces.
In this paper it would be enough to consider 
quasiharmonic functions, but as it may be of
independent interest we obtain 
the Liouville theorem in this generality.

Quasisuperharmonic functions were introduced by 
Kinnunen--Martio~\cite[Definition~7.1]{KiMa03}.
The following definition will be convenient here.
By Theorem~7.10 in~\cite{KiMa03} it is equivalent to the definition in~\cite{KiMa03}.
Also, a function is superharmonic if and only if it is $1$-quasisuperharmonic, 
by Theorem~6.1 in Bj\"orn~\cite{ABsuper} (or \cite[Theorem~9.24]{BBbook}).

\begin{deff} \label{deff-quasisuperharm}
Let $Q \ge 1$.
A function $u \in \Nploc(\Om)$ is a
\emph{$Q$-quasisuperminimizer}
in $\Om$ if 
\[ 
      \int_{\phi \ne 0} g^p_u \, d\mu
           \le Q \int_{\phi \ne 0} g_{u+\phi}^p \, d\mu
           \quad \text{for all } 0 \le \phi \in \Np_0(\Om).
\] 
A function $u:\Om \to (-\infty,\infty]$ 
is  \emph{$Q$-quasisuperharmonic} in $\Om$
if 
\begin{enumerate}
\item 
$u$ is \emph{$\essliminf$-regularized}, i.e.\
\begin{equation} \label{eq-essliminf-reg}
u(x)=\essliminf_{y \to x} u(y) \quad \text{for } x \in \Om,
\end{equation}
\item 
 $u$ is not identically $\infty$ in any component of $\Om$, and
\item 
$\min\{u,k\}$ is a $Q$-quasisuperminimizer in $\Om$ for
every $k \in \R$.
\end{enumerate}
\end{deff}

It is not difficult to see that a function $u$ is $Q$-quasiharmonic 
if and only if 
both $u$ and $-u$ are $Q$-quasisuperharmonic.

It was observed in~\cite{KiMa03} that the proof
of the weak Harnack inequality for quasiminimizers in
Kinnunen--Shan\-mu\-ga\-lin\-gam~\cite{KiSh01} holds for 
quasisuperminimizers as well. 
This then directly leads to the Liouville theorem
(under global assumptions) 
in the following form, as well as to
the strong minimum principle
(whose proof holds also under our general local assumptions,
see \cite[Section~10]{BBsemilocal}). 

\begin{thm} \label{thm-Liouville-global}
\textup{(The Liouville theorem under global assumptions,
\cite{KiMa03}, \cite{KiSh01})}
Assume that $\mu$ is globally doubling and supports a global \p-Poincar\'e
inequality.
Then every  nonnegative quasisuperharmonic function in $X$ is constant.
\end{thm}

In contrast to the proof of
Theorem~\ref{thm-Liouville-global} in~\cite{KiSh01}
(which relies on the weak Harnack inequality),
our proof of Theorem~\ref{thm-Liouville-para}
is completely elementary and only relies on the
definition of \p-parabolicity and 
properties of 
quasisuperharmonic functions.
We will use the 
following lemma, which will also be used 
to prove Theorems~\ref{thm-main-char-intro} 
and~\ref{thm-LL-F}.
We say that an unbounded open set $\Om$ is a \emph{\p-parabolic set} if 
there are $x_0 \in X$ and $\eta_j \in \Lip_c(X)$, $j=1,2,\dots$\,,
such that $0\le \eta_j \le1$ everywhere,
$\eta_j=1$ in $B(x_0,j)$ and 
\begin{equation}   \label{eq-par-set-int}
\int_\Om g_{\eta_j}^p \,d\mu \to 0, \quad \text{as } j\to\infty.
\end{equation}
An unbounded open set is a  \emph{\p-hyperbolic set
if it is not a \p-parabolic set.}
It is easily verified that $X$ is \p-parabolic in the sense 
of Definition~\ref{def-p-par} if and only if it is a \p-parabolic set.

\begin{lem}   \label{lem-quasi-Liouville-Om}
Assume that $\bdy\Om = \{x_0\}$ and that $\Om$ is either
bounded or a \p-parabolic set.
Let $u$ be a 
quasisuperharmonic function in $\Om$ that is bounded from below,
and let
\begin{equation}   \label{eq-def-u-x0}
u(x_0)=\liminf_{\Om\ni y\to x_0} u(y).
\end{equation}
Then $u(x_0)<\infty$ and $u\ge u(x_0)$ in $\Om$.

In particular, if $u$ is a bounded quasiharmonic function in $\Om$, such that 
\begin{equation}   \label{eq-lim-u-x0}
\lim_{\Om\ni y\to x_0} u(y)\quad \text{exists},
\end{equation}
then $u$ is constant.
\end{lem}

Example~\ref{ex-circle-new} with $K=\{-1\}$ 
shows that the limit \eqref{eq-lim-u-x0}
need not exist even for 
bounded \p-harmonic functions $u$ when $\Om$ is bounded.
Example~\ref{ex-factorial2} shows the same both when
$\Om$ is bounded and a \p-parabolic set, depending on the choices in the construction.

\begin{proof}
Assume to start with that $u(x_0)<\infty$.
By replacing $u$ with a positive multiple of $u+1-\inf_{\Om} u$,
we can assume that $u>0$ and that $u(x_0)=1$. 
Let $0 < \de <1$ and  $v=\min\{u+\de,1\}\le1$,  which is a 
bounded quasisuperharmonic function and thus also a quasisuperminimizer.
Moreover, by \eqref{eq-def-u-x0}, there is $\rho>0$ such that $v=1$ 
in $\clOm \cap B(x_0,\rho)$.
(Note that $\clOm = \Om\cup\{x_0\}$.)

Let $0 < \eps< 1$  and $r=1/\eps $.
Since $\Om$ is bounded or  a \p-parabolic set, there are $R>r$
and $\eta\in \Lip_c(B(x_0,R))$ such that $0\le \eta\le1$ everywhere,
$\eta=1$ in $B(x_0,r)$ and 
\begin{equation*} 
\int_\Om g_\eta^p \,d\mu < \eps.
\end{equation*}

Let  $\phi = (\eta-v)_+$.
Then $\phi=0$ in $\clOm \cap B(x_0,\rho)$
and in $\Om\setm B(x_0,R)$. 
As $\bdy \Om =\{x_0\}$ and $v \in \Nploc(\Om)$, we conclude that 
$\phi \in \Np_0(\Om)$ 
and $0 \le \phi\le 1$. 
Note that
\[
v+\phi = \eta \quad \text{whenever } \phi\ne0.
\]
Inserting this into 
Definition~\ref{deff-quasisuperharm} of quasisuperminimizers 
(with $Q$ being the quasisuperminimizing constant of $v$),
we get
\[
\int_{\phi\ne0} g_v^p \,d\mu \le Q \int_{\phi\ne0} g_{v+\phi}^p \,d\mu 
= Q \int_{\phi\ne0} g_\eta^p \,d\mu< Q\eps.
\]
Since $\phi(x)>0$ for all $x\in \Om \cap B(x_0,r)$ with $v(x)<1$, 
we therefore conclude that
\[
\int_{\Om\cap B(x_0,r)} g_{(1-v)_+}^p \,d\mu 
\le \int_{\phi\ne0} g_v^p \,d\mu
< Q\eps.
\]
Letting $\eps\to0$ (and thus $r\to\infty$) implies that 
$g_{(1-v)_+}=0$ a.e.\ in $\Om$.

Since $(1-v)_+=0$ in  $\clOm \cap B(x_0,\rho)$,
extending $v$ by $v=1$ in  $X \setm \Om$, together with 
a successive use of the \p-Poincar\'e inequality on larger
and larger balls, 
implies that $(1-v)_+=0$ (that is  $v\ge 1$) a.e.\ in  $X$.
Hence, by 
the $\essliminf$-regularity \eqref{eq-essliminf-reg},
\[
  u(x)+\de \ge v(x)\ge 1 = u(x_0) \quad \text{for every } x \in \Om.
\]
Letting $\de \to 0$ proves the first statement when $u(x_0)<\infty$.

Assume now 
that $u(x_0)=\infty$.
Let $u_k:=\min\{u,k\}$, where $k \in \R$.
Then the above applies to $u_k$ and thus
$u_k(x) \ge u_k(x_0) =k$ in $\Om$. 
Letting $k \to \infty$ shows that 
$u \equiv \infty$ in $\Om$.
But this is impossible as $u$ is quasisuperharmonic.
Hence $u(x_0)<\infty$.

The second statement follows by applying the first part to  both $u$ and $-u$.
\end{proof}

\begin{proof}[Proof of Theorem~\ref{thm-Liouville-para}]
Let $\{x_j\}_{j=1}^\infty$ be a sequence such that 
\begin{equation} \label{eq-L-para} 
u(x_j)\to \sup_X u,
\quad \text{as } j\to\infty.
\end{equation}
Since $u$ is $\essliminf$-regularized
as in~\eqref{eq-essliminf-reg}, we see that
$u(x_j)=\liminf_{\Om\ni y\to x_j} u(y)$.
Lemma~\ref{lem-quasi-Liouville-Om} implies that
$u  \ge u(x_j)$ in $X$ for $j=1,2,\dots$\,.
Together with~\eqref{eq-L-para} this shows that $u \equiv \sup_X u$.
\end{proof}

The following observation is rather trivial, but will
be useful to refer to in some of our proofs.

\begin{lem} \label{lem-Liouville=>connected}
If the Liouville theorem holds for bounded $Q$-quasiharmonic functions in 
a nonempty open set $G$, then $G$ is connected.
\end{lem}

\begin{proof}
Let $G'$ be a component of $G$, then $\chi_{G'}$ is a bounded
\p-harmonic function in $G$ which is constant by the Liouville theorem.
Hence $G=G'$ is connected.
\end{proof}

\begin{proof}[Proof of Theorem~\ref{thm-main-char-intro}]
\ref{f-Li}\imp\ref{f-rem}
This is trivial: Since every function under consideration is constant,
one can just extend it as a constant function.

\ref{f-rem}\imp\ref{f-weak}
This implication is also trivial.

\ref{f-weak}\imp\ref{f-char}
Properties \ref{rs-par}--\ref{rs1+2} follow directly from 
Theorem~\ref{thm-cpt-pharm-nec}
since $\bdy G = \bdy\Om \cup \bdy K$
and $\Cp(\bdy\Om)=0$ by Theorem~\ref{thm-cpt-pharm-nec}\ref{r-1}.
As for \ref{rs-lim}, the limit~\eqref{eq-lim-ex} exists because the quasiharmonic
extension is continuous at $x_0$.

\ref{f-char}\imp\ref{f-Li}
Let $u$ be a bounded $Q$-quasiharmonic function in $G$.
Lemma~\ref{lem-G_0-open} and \ref{rs1+2} imply
that the set $G_0:=\clG\setm\{x_0\}$ is open.
By \ref{rs1+2} and Theorem~\ref{thm-removability-qharm}, $u$ has
a bounded $Q$-quasiharmonic extension $\ut$ to $G_0$.
Since 
$\ut$ is continuous in $G_0\subset \clG$, it follows from \ref{rs-lim} that
\[
   \lim_{G_0 \ni y \to x_0} \ut(y)
   =\lim_{G \ni y \to x_0} u(y)
\quad \text{exists.}
\]
As $\bdy G_0=\{x_0\}$, it now follows from 
Lemma~\ref{lem-quasi-Liouville-Om} and \ref{rs-par}
that $\ut$, and thus $u$,  is constant.
Hence \ref{f-Li} holds.
Moreover, $G$ is connected by Lemma~\ref{lem-Liouville=>connected}.
\end{proof}

\begin{proof}[Proof of Corollary~\ref{cor-remove-K-from-X-hyp}]
If $\Cp(K)=0$, then $K$ is removable by Theorem~\ref{thm-removability-qharm}.
Conversely, assume that $\Cp(K)>0$.
Then, by assumption, either \ref{rs-par} 
in Theorem~\ref{thm-main-char-intro}
fails (if $X$ is \p-hyperbolic)
or \ref{rs1+2} 
in Theorem~\ref{thm-main-char-intro}
fails (by Lemma~\ref{lem-Cp-connected}\ref{it-Cp}
if either $\Cp(X\setm \Om)>0$ or $\Cp(\{x\})=0$ for every $x \in \bdy K$).
Hence $K$ is not removable by 
Theorem~\ref{thm-main-char-intro}.
\end{proof}

\section{The Liouville theorem in open subsets
and  Theorem~\ref{thm-LL-F}}
\label{sect-Liouville-G}

\emph{Recall the general assumptions from the beginning of Sections~\ref{sect-ug}
and~\ref{sect-harm}.}

\medskip

The following observation is an easy consequence
of the extension Theorem~\ref{thm-removability-qharm}.
It can be conveniently combined with
Theorem~\ref{thm-Liouville-para} or~\ref{thm-Liouville-global}.

\begin{prop} \label{prop-L-CpF=0-new}
Assume that $G \subset X$ is open and that  $\Cp(X \setm G)=0$.
Then the Liouville theorem for bounded $Q$-quasiharmonic functions 
holds in $G$ if and only if it holds on $X$. 
\end{prop}

\begin{proof}
If $u$ is a bounded $Q$-quasiharmonic function in $G$,
then by Theorem~\ref{thm-removability-qharm}
 it has a bounded $Q$-quasiharmonic  extension to $X$,
which must be constant if the  Liouville theorem holds on $X$.
This proves one implication.

Conversely, the  Liouville theorem for $G$ implies that bounded 
$Q$-quasiharmonic functions in $X$ have to be constant 
in $G$ and hence in $X$, by continuity and the fact that 
$G$ has empty exterior (because $\Cp(X\setm G)=0$).
\end{proof}

However, there are situations when the Liouville theorem holds also for
complements of closed sets with positive \p-capacity, even in \p-hyperbolic spaces,
see Example~\ref{ex-two-R} (with a \p-hyperbolic $X$), 
Example~\ref{ex-bow-ties}(a) (with a \p-parabolic $X$)
and Example~\ref{ex-bow-ties}(b) (with a bounded $X$).

To prove the necessity in Theorem~\ref{thm-LL-F}, we will use the 
following lemma.
Typical situations are 
bow-ties (see Example~\ref{ex-bow-ties}), metric trees and metric graphs.

\begin{lem} \label{lem-Y-PI}
Assume that $G$ is an open connected set with 
$\Cp(\bdy G \setm \{x_0\})=0$, where $x_0 \in \bdy G$.
Then $Y:=\clG$, equipped with the metric $d$ and the 
measure $\mu_Y:=\mu|_Y$, is a proper connected metric space such that 
$\mu_Y$ is a locally doubling measure supporting a local \p-Poincar\'e 
inequality on $Y$.

Moreover, $\CpY(\{x_0\})=0$ if and only if $\Cp(\{x_0\})=0$, 
and this is only possible if $Y=X$.
\end{lem}

We denote balls in $Y$ by 
$B^Y=B^Y(x,r)= \{y \in Y : d(x,y)<r\}$, where 
$x\in Y$ and $r>0$,
and the Sobolev capacity in $Y$ by $\CpY$.
In this proof, we will write $a \simle b$ and $b \simge a$ 
if there is an implicit comparison constant $C>0$ such that $a \le Cb$, 
and $a \simeq b$ if $a \simle b \simle a$.
The implicit comparison constants are allowed to depend on the 
fixed data.

\begin{proof}
By Lemma~\ref{lem-G_0-open}, the set 
$G_0:=\clG\setm\{x_0\}$ is open.
As $G$ is connected so is $G_0$.
We can therefore replace $G$ by $G_0$, 
or in other terms assume that $\bdy G=\{x_0\}$,  in the rest of the proof.

It is clear that
$Y$ is proper and connected and that $0< \mu_Y(B^Y)< \infty$ 
for all balls $B^Y$ in $Y$.
The local doubling and local \p-Poincar\'e inequality for $\mu_Y$ 
near
an arbitrary point  $x \in G$ follow from the corresponding property
for $\mu$ on $X$.
It remains to show these properties near $x_0$.

Let $r_0\le \min\{1,\tfrac13 \diam Y\}$ be so small that
$\mu$ is doubling and supports a \p-Poincar\'e inequality
within $B(x_0,5r_0)$.
Then there is a quasiconvexity constant  $L \ge 1$
such that any pair of points $x,y\in B(x_0,r_0)$ can be connected by 
an $L$-quasiconvex curve
in $X$ of length $\ell \le Ld(x,y)$,
cf.\ Lemma~4.9 in Bj\"orn--Bj\"orn~\cite{BBsemilocal}.

First, we show that
\begin{equation} \label{eq-ball-in-Y}
    B(y,r) \subset Y
\quad \text{if } y \in Y,  \ Lr \le d(y,x_0)
\text{ and } B(y,r) \subset B(x_0,r_0).
\end{equation}
Indeed, if $z \in B(y,r) \setm Y$, then there is an 
$L$-quasiconvex curve from $y$ to  $z$
in $X$ of length $\ell \le Ld(y,z)$. 
As $\bdy G=\{x_0\}$ this curve must pass through $x_0$.
Thus
\[
d(y,x_0) \le \ell \le L d(y,z) < L r \le d(y,x_0),
\]
which is a contradiction.
Hence no such $z$ exists and $B(y,r) \subset Y$.

In order to prove that $\mu_Y$ is doubling within $B^Y(x_0,r_0)$
consider a ball $B^Y(x,r) \subset B^Y(x_0,r_0)$.
As $Y$ is connected and $Y \setm B(x_0,r_0) \ne \emptyset$ 
(since $\diam Y \ge 3r_0$), we see that $r \le 2r_0$ and 
thus $B(x,r) \subset B(x_0,5r_0)$. 
We shall show that 
\begin{equation} \label{eq-muY-B}
\mu_Y(B^Y(x,r)) \simeq \mu(B(x,r)).
\end{equation}
The $\le$ inequality is trivial.
For the $\simge$ inequality, we consider two cases.

\emph{Case}~$1$. $r \le 2d(x,x_0)$.
In this case it follows from \eqref{eq-ball-in-Y}
and the doubling property of $\mu$ within $B(x_0,5r_0)$
that
\begin{equation*} 
\mu(B(x,r)) \simle  \mu(B(x,r/2L))  
 =  \mu_Y(B^Y(x,r/2L))    \le \mu_Y(B^Y(x,r)).
\end{equation*} 

\emph{Case}~$2$. $r \ge 2d(x,x_0)$.
By the connectedness of $Y$ (and the fact that $\diam Y \ge 3r_0 > r$), there
is $y \in Y$ such that $d(y,x_0)=\tfrac14 r$.
Then $d(x,y) \le \tfrac34 r$ and hence another
use of~\eqref{eq-ball-in-Y}, together with the doubling 
property of $\mu$ within $B(x_0,5r_0)$, gives that
\begin{equation*} 
\mu(B(x,r)) 
\le \mu(B(y,2r)) 
\simeq \mu(B(y,r/4L))  
= \mu_Y(B^Y(y,r/4L)) 
  \le  \mu_Y(B^Y(x,r)).
\end{equation*}

Thus in both cases, \eqref{eq-muY-B} holds.
Since $\mu$ is doubling within $B(x_0,5r_0)$, it directly follows
that $\mu_Y$ is doubling within $B^Y(x_0,r_0)$.

Now we turn to the \p-Poincar\'e inequality near $x_0$.
Let $\la$ be the dilation constant in the \p-Poincar\'e inequality
for  $\mu$ within $B(x_0,5r_0)$. 
Consider a ball $B^Y=B^Y(x,r) \subset B^Y(x_0,r_0)$ and let  $B=B(x,r)$. 
It was shown above that $B(x,r) \subset B(x_0,5r_0)$.
Let $u$ be an integrable function in $\la B^Y$ and 
$g$ be a \p-weak upper gradient of $u$ in $\la B^Y$.
Approximating $u$ by its truncations $\max\{\min\{u,k\},-k\}$ 
at levels $\pm k$ and using 
dominated convergence shows that it is enough
to prove the Poincar\'e inequality for bounded $u$
(cf.~\cite[Proof of Proposition~4.13]{BBbook}).
Letting $u=u(x_0)$ and $g=0$ in $\la B \setm Y$
we see that $g$ is a \p-weak upper gradient of $u$ in $\la B$.

Since $\mu$ supports a \p-Poincar\'e inequality within $B(x_0,5r_0)$,
we conclude, using \eqref{eq-muY-B}, that
\begin{equation*} 
        \vint_{B^Y} |u-u_{B}| \,d\mu_Y
       \simle        \vint_{B} |u-u_B| \,d\mu 
        \simle  r \biggl( \vint_{\lambda B} g^{p} \,d\mu \biggr)^{1/p}
        \le  r \biggl(  \vint_{\lambda B^Y} g^{p} \,d\mu_Y \biggr)^{1/p}.
\end{equation*}
A standard argument using the triangle inequality
allows us to replace $u_{B}$ in the left-hand side by $u_{B^Y}$
at the cost of an extra factor~2 in the right-hand side.

For the last part, the inequality $\CpY(\{x_0\}) \le \Cp(\{x_0\})$
is obvious.
Conversely, assume that $\CpY(\{x_0\}) = 0$.
Recall that we assume that $\bdy G = \bdy Y= \{x_0\}$, see the beginning of the proof.
Since $\CpY(\{x_0\})=0$, there are \p-almost no curves in $Y=\clG$ 
hitting $x_0$ (see \cite[Proposition~1.48]{BBbook}).
Thus $g \equiv 0$ is a \p-weak upper gradient of 
$\chi_G$ 
in $X$.
The local \p-Poincar\'e inequality on $B(x_0,r_0)$ then implies that $\chi_G$
is a.e.\ constant in $B(x_0,r_0)$ and thus
$B(x_0,r_0)\subset Y$. 
Using that $\bdy Y=\{x_0\}$ and by the connectedness of $X$, we must have $Y=X$.
In particular, $\CpY=\Cp$.
\end{proof}

\begin{proof}[Proof of Theorem~\ref{thm-LL-F}]
The sufficiency is shown as in the proof of 
Theorem~\ref{thm-main-char-intro} \ref{f-char}\imp\ref{f-Li}.
For the necessity, assume
that the Liouville theorem holds for bounded
$Q$-quasiharmonic functions in $G$.
Note first that
$\Cp(\bdy G)>0$, by Lemma~\ref{lem-Cp-connected}\ref{it-Cp}, 
and thus $\bdy G \ne \emptyset$.

\ref{Lio-cap}
Assume, for a contradiction, that $\Cp(\bdy G \setm \{x\})>0$
for every $x\in \bdy G$.
By the Kellogg property (Theorem~\ref{thm-Kellogg}), 
there is a regular boundary point
$x_0 \in \bdy G$ for~$G$.
As $\Cp(\bdy G \setm \{x_0\})>0$ there is, again by the Kellogg property,
another regular boundary point $x_1 \in \bdy G \setm \{x_0\}$ for~$G$.

Let  $f(x):=(1-d(x,x_1)/d(x_0,x_1))_+$ (with $f(\binfty)=0$).
Then $f\in \Lip_c(X)$ and the upper Perron solution
$u=\uP_{G} f$ is a bounded \p-harmonic function in $G$.
Since $x_0$ and $x_1$ are regular, we see that 
\[
\lim_{G\ni y\to x_j} u(y) = f(x_j)=j, 
\quad j=0,1.
\]
In particular, $u$ is nonconstant,
which contradicts the Liouville theorem.
Thus \ref{Lio-cap} holds.

\ref{Lio-lim}
By the assumed Liouville theorem, $u$ is constant and thus the limit exists.

\ref{Lio-par}
Assume, for a contradiction, that $G$ is  
neither bounded nor a \p-parabolic set.
By the already proved part~\ref{Lio-cap}, we know that 
$\Cp(\bdy G \setm \{x_0\})=0$ for some $x_0\in \bdy G$.
In particular, $\Cp(\{x_0\})>0$ since $\Cp(\bdy G)>0$.
As $G$ is connected 
(by Lemma~\ref{lem-Liouville=>connected}),
Lemma~\ref{lem-Y-PI} implies  that
$Y:=\clG$ is a proper connected metric space (with metric $d$) and that 
$\mu_Y:=\mu|_Y$ is a locally doubling measure supporting a local \p-Poincar\'e 
inequality on $Y$, and moreover that $\CpY(\{x_0\})>0$.
Since $G$ is unbounded and not a \p-parabolic set, it
follows directly from~\eqref{eq-par-set-int} that $Y$ is \p-hyperbolic.
Hence, by Theorem~1.1 in Bj\"orn--Bj\"orn~\cite{BBglobal} applied to $Y$,
there is a  Green function $v$ in $Y$ with singularity at $x_0$.
In particular, $v$ is a nonconstant
positive bounded \p-harmonic function in $G\subset Y\setm\{x_0\}$
such that
\[
v(x_0) = \lim_{G \ni y\to x_0}v(y) = \max_{Y} v < \infty
\quad \text{and} \quad
\inf_Y v=0.
\]
(The boundedness of $v$
follows from Lemma~9.5 in~\cite{BBglobal}
since $\CpY(\{x_0\})>0$.)
This contradicts the Liouville theorem. 
Hence, \ref{Lio-par} holds.
\end{proof}

\section{Local connectedness and 
Proposition~\ref{prop-not-locconn}}
\label{sect-loc-conn}

\emph{Recall the general assumptions from the beginning of Sections~\ref{sect-ug}
and~\ref{sect-harm}.}

\medskip

We shall now prove
Proposition~\ref{prop-not-locconn}, i.e.\ 
that local connectedness is required for removability.
First, recall the definition.

\begin{deff} \label{deff-lc}
A metric space $Y$ is \emph{locally connected at $x\in Y$}
if for every $r>0$ there is a connected neighbourhood
(not necessarily open) $V \subset B^Y(x,r)$ of $x$.

$Y$ is \emph{locally connected} if it is locally connected at
every point $x \in Y$.
\end{deff}

If $Y$ is locally connected then every component of an open set
in $Y$ is open.
As mentioned at the beginning of Section~\ref{sect-harm}, 
$X$ is  locally quasiconvex and in particular locally connected,
under our general assumptions.

For removability it is local connectedness at a boundary point
that is relevant. This is defined in the following way.

\begin{deff}\label{deff-loc-conn-bdy}
$\Om$ is \emph{locally connected at $x_0\in \bdy \Om$} if
for every $r>0$ there is an open set $V\subset B(x_0,r)$ 
such that $x_0 \in V$
and $V \cap \Om$ is connected.
\end{deff}

\begin{proof}[Proof of Proposition~\ref{prop-not-locconn}]
Since $G$ is not locally connected at $x_0$ there
is $r>0$ such that $U \cap G$ is disconnected for
every open set $U$ 
with $x_0 \in U \subset B:=B(x_0,r)$.
We shall
find two disjoint open sets $V_0$ and $V_1$ such 
that $V_0 \cup V_1 = B \cap G$ and $x_0\in \bdy V_0 \cap \bdy V_1$.
There are two possibilities:
Either 
(a) $B \cap G$ has a component $V_0$ such that $x_0 \in \bdy V_0$,
or (b) there is no such component.

In case (a), we let 
$V_1=(B\cap G) \setm V_0$, which is nonempty
by the choice of $r>0$.
Assume that $x_0 \notin \bdy V_1$.
Let $W$ be the component of $B(x_0,\dist(x_0,V_1))$ containing $x_0$.
(Recall that balls need not be connected.)
As $X$ is locally connected, $W$ is open and thus 
$V:=W \cup V_0$ would be
an open connected set such that 
$V \cap G=V_0$ is connected and  $x_0 \in V \subset B$,
contradicting the choice of $r>0$.
Therefore, $x_0 \in \bdy V_1$.

In case (b), $B \cap G$ must have infinitely many components 
$V'_0,V'_1,\dots$ (as $X$ is Lindel\"of, they  
can only be countably many).  
Since $x_0\in \bdy G$, we can find a subsequence $\{V'_{j_k}\}_{k=1}^\infty$
such that $\dist(x_0,V'_{j_k})\searrow0$ as $k\to\infty$. 
Let $V_0=\bigcup_{k=1}^\infty V'_{j_{2k}}$ and $V_1=(B \cap G)\setm V_0$.
Then $x_0 \in \bdy V_0\cap \bdy V_1$.

Note that for every open connected set $U\ni x_0$, 
\[
\Cp(U \setm V_0) \ge \Cp(U \cap V_1) >0,
\]
because $x_0\in \bdy V_1$ and thus $U \cap V_1$ is a nonempty
open set.
Letting $U$ be the connected component of $B(x_0,\rho)$, $\rho>0$,
that contains $x_0$,  we 
therefore conclude from Lemma~\ref{lem-Cp-connected} 
that 
\begin{equation} \label{eq-Cp-bdy-V0}
\Cp(B(x_0,\rho)\cap \bdy V_0)>0 \quad  \text{for all } \rho>0.
\end{equation}

Next let 
\[
   \phi(x)=\begin{cases}
   (1-d(x,x_0)/r), & \text{if } x \in V_1,\\
   0, & \text{if } x \in G \setm V_1.
   \end{cases}
\]
Then $\phi \in C(G)$ with $g_\phi\le (1/r) \chi_{B(x_0,r)}$
and hence $\phi \in \Np(G)$ and $H_G \phi \in \Dp(G)$.
Note that the Sobolev solution $H_G \phi$ 
(defined in Definition~\ref{def-harm-ext}) exists because 
$\Cp(X \setm G)\ge \Cp(B \cap \bdy V_0)>0$,
by~\eqref{eq-Cp-bdy-V0}. 

We shall now localize $H_G \phi$ 
and show that it has different limits towards $x_0$
from $V_0$ and $V_1$. 
Let $\eta \in \Lipc(B)$ be a Lipschitz cut-off function such that
$\eta=1$ in $\tfrac12 B$ and $0 \le \eta \le 1$ everywhere.
Then $\eta (\phi- H_G \phi) \in \Np_0(B \cap G)$ and in particular,
$\eta (\phi- H_G \phi) \in \Np_0(V_j)$, $j=0,1$, since there are no curves 
between $V_0$ and $V_1$. 
Hence, $H_G\phi\in\Np(V_j)$ is the Sobolev solution of the
Dirichlet problem in $V_j$ with boundary values
\[
f:= H_G \phi + \eta (\phi- H_G \phi)
   = \eta \phi + (1-\eta) H_G \phi,
\]
that is
\begin{equation} \label{eq-HG=HVj}
H_G \phi    =  H_{V_j} f \quad \text{in } V_j,  \quad j=0,1.
\end{equation}
Since $\eta=1$ in $\tfrac12 B$ and $0\le H_G\phi\le 1$, we have
\begin{equation}   \label{eq-bound-f} 
0\le f \le 1-\eta
\le \frac{2d_{x_0}}{r} \quad \text{in } V_0
\qquad \text{and} \qquad 
1\ge f \ge \eta\phi \ge 1- \frac{2d_{x_0}}{r} \quad \text{in } V_1,
\end{equation}
where $d_{x_0}(x) := d(x,x_0)$.

The Kellogg property (Theorem~\ref{thm-Kellogg}) 
and \eqref{eq-Cp-bdy-V0}
provide  us with (not necessarily distinct) regular boundary
points $y_k \in B(x_0,2^{-k})\cap \bdy V_0$.
Then there are $x_k\in B(y_k,2^{-k})\cap V_0$ such that
\begin{equation}   \label{eq-choose-xj}
0\le P_{V_0} d_{x_0} (x_k) \le d_{x_0}(y_k) + 2^{-k}
\le 2^{1-k}.
\end{equation}
Since $H_{V_0} d_{x_0}=P_{V_0} d_{x_0}$ (by 
Theorem~6.1
in Bj\"orn--Bj\"orn--Shan\-mu\-ga\-lin\-gam~\cite{BBS2} or
\cite[Theorem~10.15]{BBbook}),
the comparison principle for Sobolev $H$-solutions
\cite[Lemma~8.32]{BBbook},
together with \eqref{eq-HG=HVj}--\eqref{eq-choose-xj},  then implies
that
\[
\liminf_{V_0\ni y\to x_0} H_G \phi (y)
= \liminf_{V_0\ni y\to x_0} H_{V_0} f (y)
\le\lim_{k\to\infty} H_{V_0} f(x_k) 
\le \lim_{k\to\infty} \frac{2 P_{V_0} d_{x_0} (x_k)}{r} =0.
\]
A similar argument for $V_1$ gives that 
\[
\limsup_{V_1\ni y\to x_0} H_G \phi (y)=1,
\]
which shows that the bounded \p-harmonic function $H_G \phi$ in $G$
does not have any continuous extension to $x_0$. 
The nonremovability conclusion then follows since a quasiharmonic extension
would have to be continuous at $x_0$.
\end{proof}

\section{Sequential annular quasiconvexity and 
\texorpdfstring{\eqref{eq-seq=>lim=>loc-conn}}{(1.2)}} 
\label{sect-annular}

\emph{Recall the general assumptions from the beginning of Sections~\ref{sect-ug}
and~\ref{sect-harm}.}

\medskip

Annular quasiconvexity of the space $X$ is an important
condition in metric space analysis,
see e.g.\ Korte~\cite[Theorem~3.3]{Korte07}.
In this section, we show that 
annular quasiconvexity at a boundary point $x_0 \in \bdy V$ of an open set $V$ is
a useful condition for removability results.
We do not know if such a condition has been used earlier
to study removable sets.

\begin{deff}\label{def:ann-qcvx}
An open set $V$ is \emph{sequentially  annularly quasiconvex at $x_0\in \bdy V$}
if there is $\La >1$ 
and a sequence of radii $r_k \searrow 0$, $k=0,1,2,\dots$\,,
such that the following conditions hold:
\begin{enumerate}
\renewcommand{\theenumi}{\textup{(\roman{enumi})}}%
\item
$\bdy V\cap B(x_0,2\La r_0) = \{x_0\}$,
\item
for every $k=0,1,2,\dots$\,,    
each pair of points $x,y  \in V\cap \bdy B(x_0,r_k)$
can be connected by a curve 
\[
\ga \subset V\cap (B(x_0,\La r_k)\setm B(x_0,r_k/\La))
\quad \text{with length } \ell_\ga \le \La d(x,y).
\]
\end{enumerate}
It is convenient to extend this notion to $x_0\in V$ 
in the following way:
We say that $V$  is 
\emph{sequentially annularly quasiconvex at $x_0\in V$} 
if $V \setm \{x_0\}$ is so.
\end{deff} 

In addition to the obvious situation $V=\R^n\setm \{x_0\}$, $n \ge 2$,
and manifolds as in Corollary~\ref{cor-mfld},
a typical situation of sequential  annular quasiconvexity is when
$X$ is a bow-tie as in Example~\ref{ex-bow-ties}(a)--(c),
$V=([0,\infty)^2 \cap X) \setm\{0\}$ and $x_0=0$.

Sequential annular quasiconvexity makes it possible to prove the following
result for limits of quasiharmonic
functions, cf.\ \eqref{eq-seq=>lim=>loc-conn}.

\begin{lem} \label{lem-lim-ann-quasiconvex}
If $V$ is open and sequentially 
annularly quasiconvex at $x_0\in \bdy V$, then 
\[
   \lim_{V \ni y \to x_0} u(y) \quad \text{exists}
\]
for every  bounded quasiharmonic function $u$ in $V$.
\end{lem}

To prove Lemma~\ref{lem-lim-ann-quasiconvex} we will
use the following Harnack inequality on spheres
(restricted to $V$)
whose proof is based on the sequential annular quasiconvexity and
the Harnack inequality from
Kinnunen--Shan\-mu\-ga\-lin\-gam~\cite{KiSh01}.

\begin{lem} \label{lem-Harnack-on-spheres}
\textup{(Harnack inequality on restricted spheres)}
Assume that $V$ is open and
sequentially annularly quasiconvex at $x_0$, 
with parameters $\La$ and $r_k \searrow 0$.
Let  $u$ be a nonnegative $Q$-quasiharmonic function in $V$.
Then for every $k=0,1,\dots$\,,
\begin{equation*} 
M(r_k) \le C m(r_k),
\end{equation*}
where 
\begin{equation}   \label{eq-def-m-M}
M(r):= \sup_{V \cap \bdy B(x_0,r)} u
\qquad \text{and} \qquad 
m(r) := \inf_{V \cap\bdy B(x_0,r)} u,
\end{equation}
and the constant $C$ depends on $Q$
but is independent of $u$ and $k$.
\end{lem}

Note that  the sequential
annular quasiconvexity in 
Lemmas~\ref{lem-lim-ann-quasiconvex} and~\ref{lem-Harnack-on-spheres}
cannot be replaced by local connectedness,
see Example~\ref{ex-factorial2}.

\begin{proof}
Fix $k=0,1,2\dots$ and write $r=r_k$.
For every $x,y\in \bdy B(x_0,r)$, let $\ga$ be an arc length parameterized
curve of length $\ell_\ga \le 2\La r$, 
provided 
by the sequential annular quasiconvexity, so that 
\begin{equation}   \label{eq-pick-curve}
\ga(0)=x, \quad \ga(\ell_\ga)=y \quad \text{and} \quad
\ga \subset V \cap (B(x_0,\La r) \setm B(x_0,r/\La)).
\end{equation}
We shall show that $\ga$ can be covered by 
a chain with bounded length consisting of balls $B_j$ 
admitting the standard Harnack inequality
for quasiharmonic functions in~$V$.

Let $\la$ be the dilation constant in the \p-Poincar\'e inequality
within $B(x_0,2 \La r_0)$.
Lemma~4.9 in Bj\"orn--Bj\"orn~\cite{BBsemilocal} shows that
there is a quasiconvexity constant  $L \ge 1$
such that any pair of points $z,z'\in B(x_0,2 \La r_0)$ can be connected by a 
curve 
in $X$ of length $ \le Ld(z,z')$.
For $\rho=r/50\la \La L$  and integers 
$0\le j \le \ell_\ga/\rho \le 100\la \La^2 L$,
let $x_j=\ga(j\rho)$ and $B_j=B(x_j,\rho)$.
Then $50 \la B_j  \subset B(x_0,2 \La r_0)$.

We shall show that also $50 \la B_j  \subset V$.
Indeed, if $z\in 50 \la B_j \setm V$ then 
there is an $L$-quasiconvex curve $\ga_j$  
from $z$ to $x_j$ of length $\ell_{\ga_j}   \le Ld(z,x_j) < 50\la L\rho =r/\La$,
which must contain a point $z' \in \ga_j   \cap \bdy V$.
Hence, using \eqref{eq-pick-curve} and that $z',x_j\in\ga$,
we have
\[
   d(z',x_0) 
 \le d(z',x_j)+d(x_j,x_0) 
  < \ell_{\ga_j} +\La r
< 2 \La r.
\]
Thus $z' \in \bdy V \cap B(x_0,2\La r) = \{x_0\}$,
by assumption,  i.e.\ $z'=x_0$.
Therefore, another use of \eqref{eq-pick-curve} and $z',x_j\in\ga$
gives
\[
\frac{r}{\La} \le d(x_0,x_j)  
= d(z',x_j)
\le \ell_{\ga_j} < \frac{r}{\La},
\]
which is a contradiction. 

Thus, $50 \la B_j \subset V \cap B(x_0,2\La r)$ and the Harnack inequality 
(\cite[Corollary~7.3]{KiSh01} and \cite[p.~190]{BBsemilocal}) 
holds for $B_j$,
with constant $C'$ depending on $Q$, but independent of $j$.
(The formulation of the Harnack inequality in \cite{KiSh01} is missing 
a factor $\la$,
see \cite[Section~10]{BMarola} or \cite[Section~8.4]{BBbook}.)
Using that $B_j\cap B_{j+1}\ne\emptyset$ for all $j$,
we then have that
\[
\sup_{B_j} u \le C' \inf_{B_j} u \le C' \sup_{B_{j+1}} u,
\]
and hence by iteration, $u(x) \le C u(y)$,
where $C=(C')^{1+100\la \La^2 L}$.
Taking the supremum and the infimum over all $x\in \bdy B(x_0,r)$ and
all $y\in \bdy B(x_0,r)$, respectively, concludes the proof.
\end{proof}

\begin{proof}[Proof of Lemma~\ref{lem-lim-ann-quasiconvex}]
Let $\La$ and $r_k \searrow 0$ 
be the parameters in the sequential annular quasiconvexity condition.
Let $\eps>0$ and define
$m(r)$ and $M(r)$ as in~\eqref{eq-def-m-M}.
The strong maximum principle (Theorem~\ref{thm-max}) implies that 
the sequences $\{m(r_k)\}_{k=0}^\infty$ and $\{M(r_k)\}_{k=0}^\infty$
are monotone sequences, and that for each~$k$,
\[
\min\{m(r_k),m(r_{k+1})\} \le u \le \max\{M(r_k),M(r_{k+1})\}
\quad \text{in } V \cap (B(x_0,r_{k}) \setm B(x_0, r_{k+1})).
\]
Hence,
\[
m:= \liminf_{V \ni  y\to x_0} u(y) =\lim_{k \to \infty} m(r_k)
\quad \text{and} \quad
\limsup_{V \ni  y\to x_0} u(y) =\lim_{k \to \infty} M(r_k).
\]
Moreover,     $u\ge m-\eps$ in $V\cap B(x_0,r_k)$ 
for large enough $k$.
The Harnack inequality on restricted spheres (Lemma~\ref{lem-Harnack-on-spheres})
therefore implies that for all sufficiently large~$k$,
\[
M(r_k) - (m-\eps) \le C (m(r_k) -(m-\eps)).
\]
Letting $k\to \infty$ and then $\eps\to0$ yields
\[
\limsup_{V \ni  y\to x_0} u(y)=
\lim_{k\to\infty} M(r_k) 
\le  m + (C-1)\eps \to m
=\liminf_{V \ni  y\to x_0} u(y).
\qedhere
\]
\end{proof}

\begin{proof}[Proof of Corollary~\ref{cor-X-one-pt}]
If $\Cp(\{x_0\})=0$, then 
the removability follows from Theorem~\ref{thm-removability-qharm}.
If instead $\Cp(\{x_0\})>0$,
then the equivalence follows directly from 
Theorem~\ref{thm-main-char-intro} together with Lemma~\ref{lem-lim-ann-quasiconvex}.
\end{proof}

\begin{proof}[Proof of Corollary~\ref{cor-mfld}]
If $\Cp(K)=0$, then 
the removability follows from Theorem~\ref{thm-removability-qharm}.
On  the other hand, if $\Cp(K)>0$ then the equivalence
follows directly from  Theorem~\ref{thm-main-char-intro}
together with \eqref{eq-seq=>lim=>loc-conn}
(or Lemma~\ref{lem-lim-ann-quasiconvex}),
because $\Om \setm \{x_0\}$ is 
sequentially annularly connected at $x_0$.
\end{proof}

\begin{example} \label{ex-Rn-1/x-loc}
Let $w>0$ be a weight on $\R^n$, with $p=n\ge2$, such that
$w$ is continuous except at $0$ and 
$w(x)=1/|x|$ for $x$ sufficiently close to $0$.
Then $d\mu=w\,dx$ is a locally doubling measure
supporting a local $1$-Poincar\'e inequality, see
Hei\-no\-nen--Kil\-pe\-l\"ai\-nen--Martio~\cite[Section~1.6]{HeKiMa}.
Moreover, $\Cpmu(\{0\})>0$ while $\Cpmu(\{x\})=0$ for $x \ne 0$,
by Example~2.12 in~\cite{HeKiMa}. 
So there are plenty of compact sets $K \subset \Rn$
such that $\Cpmu(K)>0=\Cpmu(K \setm \{0\})$
for which Corollary~\ref{cor-mfld} is relevant.
\end{example}

\section{Examples of removable sets}
\label{sect-examples}

In this section, we give several examples
of unbounded closed removable sets with infinite measure.
In 
Example~\ref{ex-two-R}(a), (c)
the removable sets even separate the space.

\begin{example}  \label{ex-bow-ties}
(Bow-ties) 

(a) 
Let 
\[
   X=\Om=\{(x_1,x_2) \in \R^2 : x_1x_2 \ge 0\},
\]
which consists of the (closed) first and third quadrants.
We equip $X$ with the Euclidean metric and
the $2$-dimensional Lebesgue measure $\mu$ and let $p>2$.
Then $\mu$ is Ahlfors $2$-regular, and it follows
from Heinonen--Koskela~\cite[Theorem~6.15]{HeKo98}
(or \cite[Example~A.23]{BBbook})
that $\mu$ supports a global \p-Poincar\'e inequality on $X$.
Let $F=(-\infty,0]^2$ and $G=\Om \setm F = [0,\infty)^2 \setm\{(0,0)\}$.
Then $G$ is annularly quasiconvex at $(0,0)$ and $G$ is a \p-parabolic
set (as $\R^2$ is \p-parabolic).

It thus follows from  Lemma~\ref{lem-quasi-Liouville-Om}
(or Theorem~\ref{thm-LL-F})
together with Lemma~\ref{lem-lim-ann-quasiconvex} 
that any bounded
quasiharmonic function in $G$ is constant.
It can then be trivially extended into a (constant) \p-harmonic function 
in $X$.
So $F$ is removable for bounded $Q$-quasiharmonic functions.

Similar arguments apply in the following two situations:
\begin{enumerate}
\stepcounter{enumi}
\item (Bounded bow-tie) 
\[
   X=\Om=\{(x_1,x_2) \in [-1,1]^2 : x_1x_2 \ge 0\}
\quad \text{and} \quad K=[-1,0]^2.
\]
\item (Half-unbounded bow-tie) 
\[
X=\Om=\{(x_1,x_2) \in [-1,\infty)^2 : x_1x_2 \ge 0\}
 \text{ and } 
   K=[-1,0]^2 \text{ or }F= [0,\infty)^2.
\]
\end{enumerate}
\end{example}

The following example gives removability for closed sets
with positive capacity in cases when
there are bounded nonconstant \p-harmonic
functions that can be extended.
In particular, part~(a) shows that 
for a closed unbounded set $F$ removability
is not equivalent to the Liouville theorem on $X \setm F$,
cf.\ Theorems~\ref{thm-main-char-intro} and~\ref{thm-LL-F}.

\begin{example} \label{ex-hyp-R}
Let $X=\R$ be equipped with the weighted measure $d\mu=w\,dx$,
where $w(x)=\max\{1,|x|\}^p$.
Then $X$ is \p-hyperbolic and the two 
intervals $(-\infty,-1)$ and $(1,\infty)$ are \p-hyperbolic sets,
sometimes called \p-hyperbolic ends.

(a)
Let $G_1=(-1,1)$ and  $F_1=X \setm G_1$. 
Assume that $u$ is \p-harmonic in $G_1$. 
Then 
$u$ is an affine function
(and $u$ is bounded).
Without loss of generality we can assume that $u(x)=x$.
Define the extension $\ut$ to $X$ 
as an odd function with
\begin{equation*} 
\ut(x)= 1 + \int_1^x \frac{dt}{t^{p/(p-1)}} 
= p- \frac{p-1}{x^{1/(p-1)}}
\quad \text{for } x\in [1,\infty).
\end{equation*}
It is easily verified that $\ut$ is a bounded \p-harmonic function in $X$,
see Bj\"orn--Bj\"orn--Shan\-mu\-ga\-lin\-gam~\cite[Lemma~6.2]{BBSliouville}.
Hence 
$F_1$ is removable for bounded \p-harmonic functions in $G_1$.
Note that \ref{r-3} and \ref{r-5} in Theorem~\ref{thm-cpt-pharm-nec}
fail, so the compact set $K$ therein cannot be replaced by an unbounded closed set.

(b)
Let $F_2=F_1$ be as in (a) and $G_2=G_1 \setm \{0\}$.
Assume that $u$ is \p-harmonic in $G_2$. 
Then the restrictions $u|_{(-1,0)}$ and $u|_{(0,1)}$
are affine functions
(and $u$ is bounded).
Without loss of generality we can assume that
\[
     \lim_{G_2 \ni y \to \pm 1} u(y)=\pm 1.
\]
Define the extension $\ut$ to $\Om_2 := G_2 \cup F_2$
by 
\[
\ut(x)= \begin{cases}
 \displaystyle 
1+ a\biggl(1-\frac{1}{x^{1/(p-1)}}\biggr)
   & \text{for } x\in [1,\infty), \\
   \displaystyle 
-1+ b\biggl(1-\frac{1}{|x|^{1/(p-1)}}\biggr)
 & \text{for }  x\in (-\infty,-1].
\end{cases}
\]
By choosing $a$ and $b$ appropriately (so that $\ut'(\pm 1)$ exist)
$\ut$ becomes a bounded \p-harmonic function in $\Om_2$.
Hence 
$F_2$ is removable for bounded \p-harmonic functions in $G_2$.
Note that \ref{r-1}--\ref{r-2}, \ref{r-3} and \ref{r-5} 
in Theorem~\ref{thm-cpt-pharm-nec}
fail, so the compact set $K$ therein cannot be replaced by an unbounded closed set.
Property~\ref{r-3b} in Theorem~\ref{thm-cpt-pharm-nec}
always hold (also when $K$
is merely relatively closed in $\Om$), by 
Lemma~\ref{lem-Cp-connected}\ref{it-Cp}.
\end{example}

Since $X$ is \p-hyperbolic, neither in (a) nor in (b) of
Example~\ref{ex-hyp-R} do we know whether
$F_j$ is removable for bounded $Q$-quasiharmonic functions
in $G_j$, $j=1,2$.

The following example (parts~(a) and~(b)) has
removability for 
bounded quasiharmonic functions in a \p-hyperbolic space.
In parts~(a) and~(c) the removable unbounded set separates the
space $X=\Om$, while in part~(b) it instead has nonunique bounded extensions.
For bounded \p-harmonic functions there 
are similar examples (of removable sets that separate and/or have nonunique extensions)
in Examples~10.2 and~10.3 in Bj\"orn~\cite{ABremove}
(or \cite[Example~12.27 and 12.28]{BBbook}).
A novelty here is that the set is removable for 
bounded $Q$-quasiharmonic functions.

\begin{example} \label{ex-two-R}
Let $X$ consist of two copies of $\R$, identified at the interval $[-1,1]$
and equipped with the length metric,
i.e.\ $d(x,y)=\ell_{\ga}$, the length of the shortest curve $\ga$ from $x$ to $y$.

One copy of $\R$ is equipped with the Lebesgue measure, while the remaining 
intervals $(-\infty,-1)$ and $(1,\infty)$ are equipped with the measure 
$d\mu(x) = |x|^p\,dx$.
Then $X$ becomes \p-hyperbolic,
while the two unweighted 
intervals $(-\infty,-1)$ and $(1,\infty)$ are \p-parabolic sets,
sometimes called \p-parabolic ends.
Let $\Om=X$ and note that by e.g.\ Theorem~\ref{thm-main-char-intro}, 
the compact interval $K=[-1,1]$ is not removable.

(a)
Let $F_1$ be the weighted real line consisting of the two
2-hyperbolic intervals together with the common interval $[-1,1]$.
We shall now see that the unbounded set $F_1$ with infinite measure is removable 
for bounded $Q$-quasiharmonic functions in $G_1:=X\setm F_1$,
even though $F_1$ disconnects $X$.

Since $G_1$ consists of the two unweighted \p-parabolic intervals,
Theorem~\ref{thm-LL-F} shows that
the only bounded quasiharmonic functions in $G_1$ 
are constant on each of these two intervals, say (without loss of generality) 
$u=-1$ in $(-\infty,-1)$ and $u=1$ in $(1,\infty)$.
Define the extension to $X$ as $\ut(x)=x$ for $x\in [-1,1]$ and 
as an odd function with
\begin{equation}   \label{eq-ext-by-int}
\ut(x)= 1 + \int_1^x \frac{dt}{t^{p/(p-1)}} 
= p- \frac{p-1}{x^{1/(p-1)}}
\quad \text{for } x\in (1,\infty)
\end{equation}
in the rest of $F_1$.
It is easily verified that $\ut$ is a bounded \p-harmonic function in $X$,
see Bj\"orn--Bj\"orn--Shan\-mu\-ga\-lin\-gam~\cite[Lemma~6.2]{BBSliouville}.
Hence, every bounded quasiharmonic function in $G_1$ extends to a bounded
\p-harmonic function in $X$.
So $F_1$ is removable for bounded $Q$-quasiharmonic functions.

Note that \ref{r-2}, \ref{r-3} and \ref{r-5} in Theorem~\ref{thm-cpt-pharm-nec}
fail, so the compact set $K$ therein cannot be replaced by an unbounded closed set.

(b)
Let $G_2=(-\infty,-1)$ in the unweighted copy of $\R$ and let $F_2=X\setm G_2$.
Then every bounded quasiharmonic function $u$ in $G_2$ is constant 
and extends trivially as a constant \p-harmonic function to $X$.
So $F_2$ is removable 
for bounded $Q$-quasiharmonic functions in $G_2$.

In addition to this trivial extension, we can first extend $u$ by 
an arbitrary constant in 
$F_2\setm F_1 = (1,\infty)$, where $F_1$ is the weighted
line from  (a), so that $u$ becomes \p-harmonic in $X\setm F_1$.
Using part (a), $u$ can then be extended to a bounded \p-harmonic function in $X$.
This shows that the \p-harmonic extension of $u$ to $X$ is not unique and can be
nonconstant even when $u$ is constant.
See Section~10 in Bj\"orn~\cite{ABremove}
(or \cite[Section~12.4]{BBbook}) for other examples with 
nonunique removability for bounded \p-harmonic functions.

(c)
This time we let $F_3=F_1$ be as in (a) and $G_3=(-2,-1) \cup (1,2)$ 
in the unweighted real line.

Assume that $u$ is \p-harmonic in $G_3$. 
Then the restrictions
$u|_{(-2,-1)}$ and $u|_{(1,2)}$ are affine functions
(and $u$ is bounded).
Without loss of generality we can assume that
\[
     \lim_{G_3 \ni y \to \pm 1} u(y)=\pm 1.
\]
Define the extension to $\Om_3:=G_3 \cup F_3$
as $\ut(x)=x$ for $x\in [-1,1]$ and 
in the hyperbolic part using~\eqref{eq-ext-by-int} by 
\[
\ut(x)= \begin{cases}
 \displaystyle 
1+ a\biggl(1-\frac{1}{x^{1/(p-1)}}\biggr)
& \text{for } x\in (1,\infty), \\
   \displaystyle 
-1+ b\biggl(1-\frac{1}{|x|^{1/(p-1)}}\biggr)
 & \text{for }  x\in (-\infty,1).
\end{cases}
\]
By choosing $a$ and $b$ appropriately (so that the nonlinear
mean-value property holds
in the limit near $\pm 1$, cf.\ \cite[Theorem~A.26]{BBbook}
or \cite[(7.5)]{BBSliouville}), $\ut$ becomes \p-harmonic in $\Om_3$.
Note that $\ut$ is bounded in $\Om_3$.
Hence $F_3$ is removable for bounded \p-harmonic functions
in $G_3$. 
In this case, we do not know if 
$F_3$ is removable for bounded \emph{quasiharmonic} functions
in $G_3$.

Here $F_3$
is a closed subset of $X$ and
$G_3$ is disconnected.
Thus this example illustrates that if we replace the compact set $K$
in Theorem~\ref{thm-cpt-pharm-nec} by a closed (with respect to $X$) unbounded set,
then 
\ref{r-1}, \ref{r-2}, \ref{r-3} and \ref{r-5} in Theorem~\ref{thm-cpt-pharm-nec}
fail in this example.
\end{example}

\section{Locally connected examples
without annular quasiconvexity}
\label{sect-Janas-ex}

The real line $\R$ and bow-ties (cf.\ Example~\ref{ex-bow-ties})
are  examples of
spaces $X$ with a globally doubling measure supporting a global \p-Poincar\'e inequality
but which are not sequentially annularly quasiconvex at a point $x_0$.
However in both cases, $X$ is disconnected by $x_0$.
There are other examples of spaces without sequential annular quasiconvexity,
but we are not aware of any such examples where
$X \setm \{x_0\}$ is locally connected at $x_0$ and $\mu$ is doubling and supports 
a \p-Poincar\'e inequality.
Here we provide such an example, which also has some
interesting properties for \p-harmonic functions.

Below we will create a metric graph $X$. 
It is a particular case of a metric space satisfying our axioms.
However, one can also consider \p-harmonic functions on ``classical'' discrete graphs
as in Holopainen--Soardi~\cite{HoSo1}, \cite{HoSo2}.
Shan\-mu\-ga\-lin\-gam~\cite{Sh-conv} showed that there is a direct
correspondence between a \p-harmonic function on a discrete graph,
and the corresponding \p-harmonic function (extended linearly to the edges)
on the corresponding metric graph, cf.\  also \cite[Section~A.5]{BBbook}.
In Bj\"orn~\cite[Section~10]{ABremove} (and \cite[Section~12.4]{BBbook}), 
metric graphs were used to construct
examples of sets removable for bounded \p-harmonic function with nonunique removability, and also
such that $G$ is disconnected by $E$.

In contrast to the studies above, in the metric graphs $X$ that we construct below,
the edges will have different lengths (i.e.\ they are not all considered to be unit intervals).
This will be reflected in the \p-harmonicity condition at the nodes, see below.

We will need the following result.
For compact $X$ it is a special
case of Theorem~1.3 in Zhou~\cite{Xiaodan19}.
For the reader's convenience we include the short proof.

\begin{prop} \label{prop-PI-arc-length}
  Assume that $X$ is geodesic 
and equipped with the arc length\/ \textup(i.e.\ normalized 
$1$-dimensional  Hausdorff\/\textup) measure $\mu$.  
If there is a constant $C>0$ such that
$\mu(B(x,r)) \le Cr$ for all balls $B(x,r)$,
then $\mu$ is Ahlfors $1$-regular and 
supports a global $1$-Poincar\'e inequality with dilation constant $\la=1$.
\end{prop}

We normalize the $1$-dimensional Hausdorff measure  so that $\mu([0,1])=1$.

\begin{proof}
Let $B=B(x,r)$ be a ball with $r < 2 \diam X$.
As $X$ is geodesic there is $y \in X$ with $d(x,y) = \tfrac14 r$,
and a geodesic $\ga_{xy}$ from $x$ to $y$
(i.e.\ $\ga_{xy}$ is a curve with length equal to $d(x,y)$).
Hence $\mu(B)  \ge \mu(\ga_{xy}) = \tfrac14 r$
and thus $\mu$ is Ahlfors $1$-regular.

Moreover, for each $z \in B$ there is a geodesic $\ga_{xz} \subset B$ from $x$ to $z$.
Let $u$ be an 
integrable function  in $B$ with an upper gradient $g$ in $B$.
Then
\[
|u(x)-u(z)| \le  \int_{\ga_{xz}} g \, ds
    = \int_{\ga_{xz}} g \,d\mu
    \le  \int_{B} g \,d\mu
\]
Hence
\[
\int_{B} |u(x)-u(z)|\,d\mu(z)
  \le \mu(B) \int_{B} g \,d\mu
  \le Cr \int_{B} g \,d\mu.
\]
Finally, a standard argument based on the triangle inequality
makes it possible to replace $u(x)$ 
on the left-hand side by $u_B$,
i.e.\ $\mu$ supports a global $1$-Poincar\'e inequality
with dilation constant $1$.
\end{proof}

\begin{example}  \label{ex-factorial2}
We shall construct
an example of a metric graph $X$ and a point $x_0=0 \in X$ 
with the following properties:
\begin{enumerate}
\item \label{e-1}
$X$ is a geodesic space consisting of a
countable union of arcs and equipped with the arc length 
(i.e.\ normalized $1$-dimensional Hausdorff) 
measure $\mu$. 
\item \label{e-2}
$\mu$ is Ahlfors $1$-regular, and in particular globally doubling.
\item \label{e-3}
$\mu$ supports a global $1$-Poincar\'e inequality with dilation $\la=1$.
\item \label{e-4}
$X \setm \{0\}$ is locally connected at $0$.
\item \label{e-5}
$X $ is not sequentially annularly quasiconvex at $0$.
\item \label{e-6} 
For each $1<p<\infty$, 
there is a bounded nonconstant \p-harmonic function $u$ in $X \setm \{0\}$
such that 
\[
 \liminf_{x\to0} u(x) < \limsup_{x\to0} u(x). 
\]
In particular, $\{0\}$ is not removable for
bounded \p-harmonic (or bounded $Q$-quasiharmonic) functions
in $X \setm \{0\}$.
\end{enumerate}

Let $\al>1$. 
Consider the weighted metric graph $X$ constructed as follows.
The points 
\[
\zpm_n=\pm \sum_{k=n+1}^\infty \frac{1}{(k!)^\al}, \quad 
n=0,1,\dots, 
\] 
are placed on the real line in the natural way and the interval 
$I:=[\zmin_0,\zpl_0]$ is 
equipped with the Euclidean distance and the arc length
measure.
The open interval between $\zpm_{n-1}$ and $\zpm_n$ will be
denoted $\Ipm_n$, it has both length and measure 
\[
l_n=  \mu(\Ipm_n) =   \frac{1}{(n!)^\al}, \quad n=1,2,\dots.
\]
Note that for $n\ge1$,
\begin{equation*}
\frac{1}{((n+1)!)^\al} \le
\zpl_n = \sum_{k=n+1}^\infty \frac{1}{(k!)^\al}
< \frac{1}{((n+1)!)^\al}  \sum_{j=0}^\infty \frac{1}{2^{\al j}}
= \frac{2^\al}{(2^\al-1)((n+1)!)^\al}.
\end{equation*}

Next, each $\zpl_n$ is connected to $\zmin_n$ by a path $\ga_n$ of
length $2l_n$, $n=1,2,\dots$\,. 
Points along the path $\ga_n$, in the direction from $\zmin_n$ to $\zpl_n$,
will be labelled by $t\in [-l_n,l_n]$ with $t=0$ corresponding
to the midpoint $z_n$ of $\ga_n$.
We equip $\ga_n$ with the length metric and the 
arc length measure.
Finally, let 
\[
X= I \cup \bigcup_{n=1}^\infty \ga_n,
\]
equipped with the length metric and the 
arc length 
measure $\mu$.
(Here $\ga_n \cap I=\{\zpm_n\}$ and $\ga_n \cap \ga_m=\emptyset$ 
for $n \ne m$.)

For other similar spaces, the
endpoints $\zmin_0$ and $\zpl_0$ of $I$ are either identified, or
connected by a curve $\ga_0$ of length $2l_0\ge 2\zpl_0$,
or extended through rays towards $\pm\infty$,
in which case $X=\R \cup \bigcup_{n=1}^\infty \ga_n$.

Note that $X$ is not locally annularly quasiconvex around $0$, since the
curves $\ga_n$ are too long, or more precisely since $l_n/\zpl_n \to \binfty$.
On the other hand $X \setm \{x_0\}$ is locally connected at $x_0$.
Thus \ref{e-1}, \ref{e-4} and \ref{e-5} are satisfied.

We now turn to the Ahlfors $1$-regularity of $m$.
Let $B=B(x,r)$ with $r < 2 \diam X$.
As $X$ is geodesic, we see that
$\mu(B(x,r)) \ge  r$ if $r \le \frac{1}{2} \diam X$,
and thus $\mu(B(x,r)) \ge  \frac{1}{4}r$ when  $r <2 \diam X$.

For the reverse inequality, we first let
\[
\Xplus=\{x \in X : \text{there is a geodesic $\ga$ from $0$ to $x$
  such that } \ga \cap (0,\infty) \ne \emptyset\}.
\]
We may assume, without loss of generality, that $x \in \Xplus$.
Let $y \in \clB$ be a point closest to $0$, i.e.
\[
      d(y,0)=\min_{z \in \clB} d(z,0).
\]
Then for every $t>0$ there are at most two points $z \in B \cap \Xplus$
with $d(y,z)=t$.
Thus $\mu(B) \le 2 \mu(B \cap \Xplus) \le 4r$,
i.e.\ $\mu$ is Ahlfors $1$-regular and \ref{e-2} has been shown.
Proposition~\ref{prop-PI-arc-length} now implies 
\ref{e-3}.

It remains to 
show \ref{e-6}.
  Let $1<p<\infty$.
  We shall construct a bounded \p-harmonic function $u$ in $X\setm \{0\}$ such
that 
\begin{equation} \label{eq-liminf<limsup}
\liminf_{x\to0} u(x) < \limsup_{x\to0} u(x).
\end{equation}
Such a function has to be linear on each curve $\ga_n$ and each subinterval $\Ipm_n$,
and in addition satisfy a certain \p-harmonicity condition
at $\zpm_n$, $n=1,2,\dots$\,, which is similar to the condition
for \p-harmonic functions on graphs with unit length edges, 
see e.g.\ \cite[Section~A.5]{BBbook}, but here we also need
to take into consideration the lengths of the curves meeting at $\zpm_n$, see below.

For this, let 
$u(z_n) = 0$, $n=1,2,\dots$\,, 
\[
u(\zpl_0) = -u(\zmin_0) =: a_0 \ge 0
 \quad 
\text{and} \quad
u(\zpl_1) =-u(\zmin_1) =  : a_1 >0
\] 
be given, with $a_1 \ge a_0$.
We shall write $a_n=u(\zpl_n)=-u(\zmin_n)$ and extend $u$
linearly to each subinterval $\Ipm_n$, as well as to each $\ga_n$, $n=1,2,\dots$\,:
\[
u(t) = a_n (n!)^\al t \quad \text{along } \ga_n.
\]
The constants $a_n$ will be determined later by a recursive formula.

To ensure that $u$ satisfies the \p-harmonicity condition at $\zpl_n$
(and by symmetry at $\zmin_n$),
consider the minimal upper gradients (derivatives) of $u$ on 
$\Ipl_{n+1}$, $\Ipl_{n}$ and $\ga_n$.
With $|\ga|$ denoting the length of a curve $\ga$, we have
\[
g_u = \frac{2a_n}{|\ga_n|} =
 a_n (n!)^\al \quad \text{on } \ga_n
\]
and
\[
g_u = \frac{(u(\zpl_n)-u(\zpl_{n-1}))}{|\Ipl_n|} = (a_n-a_{n-1})(n!)^\al
\quad \text{on } \Ipl_{n},
\] 
while $g_u=(a_{n+1}-a_n)((n+1)!)^\al$ on $\Ipl_{n+1}$.

The compatibility condition for \p-harmonicity, see e.g.\
\cite[Section~A.5]{BBbook}, thus gives us
the recursive formula at each $\zpm_n$, $n=1,2,\dots$\,,
\begin{equation} \label{eq-recursive}
\bigl( (a_{n+1}-a_n)((n+1)!)^\al \bigr)^{p-1} 
  = \bigl( (a_n-a_{n-1})(n!)^\al)^{p-1} + \bigl(a_n (n!)^\al \bigr)^{p-1},
\end{equation}
or equivalently,
\[
a_{n+1} = a_n + \frac{\bigl( (a_n-a_{n-1})^{p-1} + a_n^{p-1} \bigr)^{1/(p-1)}}{(n+1)^\al},
\quad n=1,2,\dots.
\]
Since $a_1\ge a_0$ and $a_1>0$, 
we see that $\{a_n\}_{n=1}^\infty$ is a strictly increasing 
(and thus nonconstant) sequence.
To see that it is bounded we note that as $a_{n-1}\ge0$,
\[
a_{n+1} \le a_n \biggl( 1+\frac{2^{1/(p-1)}}{(n+1)^\al} \biggr) \le \dots \le
     a_1 \prod_{k=1}^n \biggl( 1+\frac{2^{1/(p-1)}}{(k+1)^\al} \biggr)
   < a_1 \prod_{k=1}^\infty \biggl( 1+\frac{2^{1/(p-1)}}{(k+1)^\al} \biggr).
\]
The infinite product converges since $\al>1$ and hence 
\begin{equation*} 
0 <a:= \lim_{n\to\infty} a_n <\infty.
\end{equation*}
If $\zpl_0$ is identified with $\zmin_0$, we need to have
$a_0=0$.
If they are connected by a curve $\ga_0$ of length $2l_0$, we need to
choose $a_1 > a_0 >0$ so that
\[
  \frac{a_1-a_0}{l_1}=\frac{a_0}{l_0}
\]
and extend $u$ linearly along $\ga_0$.  
Finally, if they are extended through rays towards $\pm\infty$,
we need to choose $a_1=a_0>0$.
In this case we extend $u$ constantly along the two rays,
so that $u$ remains bounded.
Moreover, $X$ is \p-parabolic in this case.
 
In all cases, $u$ is bounded and
\[
\liminf_{x \to 0} u(x)=-a  < a =     \limsup_{x \to 0} u(x),
\]
i.e.\  \eqref{eq-liminf<limsup} holds.
\end{example}

\begin{example} \label{ex-factorial3}
Let $X$ be the space constructed in Example~\ref{ex-factorial2} but
with $\al=1$, i.e.\ 
\[
\zpm_n=\pm \sum_{k=n+1}^\infty \frac{1}{k!}, \quad n=0,1,\dots.
\] 
As before, $X$ is equipped with the length metric and the Lebesgue
measure, which makes $X$ Ahlfors $1$-regular.
The global $1$-Poincar\'e inequality is also
supported, by Proposition~\ref{prop-PI-arc-length}.
Depending on the choices in the construction 
(see Example~\ref{ex-factorial2}),
$X$ is either bounded or \p-parabolic.

Again, $X$ is not locally annularly quasiconvex at $0$,
but $X \setm \{0\}$ is locally connected at $0$.
Thus the properties \ref{e-1}--\ref{e-5} in Example~\ref{ex-factorial2}
are satisfied here too.
In contrast to Example~\ref{ex-factorial2} we will here
show that $\{0\}$ is removable for bounded $2$-harmonic functions in $X\setm\{0\}$.
More precisely, we will show that
if $u$ is a $2$-harmonic function in some open punctured neighbourhood $U \subset X \setm \{0\}$
of $0$, then either:
\newlength{\saveitemsep}%
\setlength{\saveitemsep}{\itemsep}%

\begin{enumerate}
\renewcommand{\theenumi}{\textup{(\roman{enumi})}}%
\setlength{\itemsep}{\saveitemsep}%
\item
$\displaystyle \lim_{x \to 0} u(x)$ exists and is finite, or 
\item
$ \displaystyle 
\limsup_{x\to0} u(x)=\infty 
\text{ and } 
\liminf_{x\to0} u(x)=-\infty.
$
\end{enumerate}

Let $u$ be a $2$-harmonic function, defined in a punctured open neighbourhood 
\[
U=\bigcup_{n=N+1}^\infty (\Ipl_n \cup \Imin_n\cup\ga_{n)}
\]
of~$0$ for some integer $N\ge1$.
We shall show that either $\lim_{U \ni x \to 0} u(x)$ exists and is finite,
or $u$ is unbounded at $0$ with
\[
\limsup_{x\to0} u(x)=\infty \quad \text{and}  \quad
\liminf_{x\to0} u(x)=-\infty.
\]

Since $u$ is linear on $\Ipm_{N+1}$, it has finite limits at $\zpm_{N}$.
We denote these limits by $\apm_{N}$.
By subtracting from $u$ the unique bounded solution $v$ of the Dirichlet problem
in $U$ with boundary values $v(\zpm_{N})=\apm_{N}$ and $v(0)=0$,
we can assume that $\apm_{N}=0$.
Note that $v$ is bounded and that $\lim_{x\to0}v(x)=0$ because 
$C_2(\{0\})>0$ and thus $0$ is
a regular boundary point (by the Kellogg property 
Theorem~\ref{thm-Kellogg}).
For $n=N+1,N+2,\dots$\,, write 
\[
\apm_n:= u(\zpm_n) \quad \text{and} \quad d_n:= \apl_n-\amin_n.
\]
The compatibility condition for $2$-harmonicity at $\zpm_n$, $n\ge N+1$,
gives as in~\eqref{eq-recursive} that    
\begin{equation}    \label{eq-apl-n+1-n}
(\apl_{n+1}-\apl_n)((n+1)!)
    = (\apl_n-\apl_{n-1})(n!)  + (\apl_n-\amin_n) \frac{n!}{2},
\end{equation}
and similarly, 
\begin{equation}    \label{eq-amin-n+1-n}
(\amin_{n+1}-\amin_n)((n+1)!)
    = (\amin_n-\amin_{n-1})(n!) + (\amin_n-\apl_n) \frac{n!}{2}.
\end{equation}
Dividing by $(n+1)!$ and subtracting \eqref{eq-amin-n+1-n} from \eqref{eq-apl-n+1-n},
we obtain for  all integers $n\ge N+1$,
\begin{equation}   \label{eq-recursion-d-n}
d_{n+1} = \biggl( 1 +\frac{2}{n+1} \biggr) d_n - \frac{d_{n-1}}{n+1},
\end{equation}
where $d_{N}=0$.
We shall now distinguish two cases.

{\bf Case~1:}
$d_{N+1}=0$.
Then \eqref{eq-recursion-d-n} clearly implies that $d_n=0$
for all $n\ge N$. 
It follows that $u$ is constant on each $\ga_n$ and linear both
in $[\zmin_{N},0]$ and in $[0,\zpl_{N}]$.
Thus the one-sided limits $\lim_{x\to0\pm}u(x)$
exist and are finite. They are also equal because $d_n=0$ for large $n$.

{\bf Case~2:}
$d_{N+1} \ne 0$. Then
we may
without loss of generality assume that $d_{N+1}>0$.
Since $d_N=0$, it 
is then easily shown by induction that for $n\ge N+1$, 
\[
d_{n+1} = \biggl( 1 +\frac{1}{n+1} \biggr) d_n + \frac{d_n-d_{n-1}}{n+1} >
d_n,
\]
and hence 
\begin{equation}   \label{eq-dn-to-infty}
d_n \ge  \prod_{k=N+2}^n\biggl( 1 +\frac{1}{k} \biggr) d_{N+1} \to \infty,
\quad \text{as }n\to\infty,
\end{equation}
i.e.\ $u$ is unbounded at $0$.

Moreover, letting $s_n=\apl_n+\amin_n$, 
we have by adding~\eqref{eq-apl-n+1-n} and~\eqref{eq-amin-n+1-n} that
\[
s_n- s_{n-1} = \frac{s_{n-1}-s_{n-2}}{n} = 
\dots = \frac{(N+1)! (s_{N+1}-s_N)}{n!}.
\]
This implies that $\{s_n\}_{n=1}^\infty$ is a Cauchy sequence
and hence the limit $\lim_ {n\to\infty} s_n$ exists and is finite.
It then follows from~\eqref{eq-dn-to-infty} that
\[
\apl_n = \tfrac12 (s_n+d_n) \to\infty
\quad \text{and} \quad
\amin_n = \tfrac12 (s_n-d_n) \to -\infty,
\quad \text{as } n\to\infty,
\]
i.e.\ $\limsup_{x\to0} u(x)=\infty$
while $\liminf_{x\to0} u(x)=-\infty$.

Finally, for removability, consider a bounded $2$-harmonic 
function $u$ in $G:= X\setm\{0\}$.  By the above, it has a (finite)
limit at $0$.  Since $X$ is bounded or $2$-parabolic, 
$u$ must be constant, by Theorem~\ref{thm-main-char-intro}.
Hence $\{0\}$ is removable for bounded $2$-harmonic functions in 
$X\setm\{0\}$.
\end{example}

\end{document}